\documentclass[11pt]{article}   

\usepackage[english]{babel}
\usepackage[margin=1.2in]{geometry}
\usepackage{graphicx}
\usepackage[pdfborder={0 0 0}]{hyperref}
\usepackage{lipsum}
\usepackage{amssymb}
\usepackage{amsmath}
\usepackage{amsthm}
\usepackage[pdftex,dvipsnames]{xcolor} 
\usepackage{dsfont}
\usepackage[colorinlistoftodos,prependcaption,textsize=small]{todonotes}
\usepackage{comment}
\usepackage{hyperref}
\usepackage{mathrsfs}
\usepackage{enumerate}
\usepackage{framed}
\usepackage{float}
\usepackage{url}
\usepackage{caption}
\usepackage{subcaption}

\usepackage[numbers,longnamesfirst]{natbib}

\newcommand{\pspace}{$(\Omega,\mathcal{F},\mathbb{P})$}

\newcommand{\R}{\mathbb{R}}
\newcommand{\rr}{\mathbb{R}}
\newcommand{\E}{\mathbb{E}}

\newcommand{\PP}{\mathbb{P}}
\newcommand{\pp}{\mathbb{P}}
\newcommand{\F}{\mathbb{F}}
\newcommand{\cF}{\mathcal{F}}

\newcommand{\one}{\mathds{1}}

\newcommand{\lambdai}{\lambda_{\infty}}

\newcommand{\Top}{^{\textrm{T}}}
\newcommand*\diff{\mathop{}\!\mathrm{d}}
\newcommand{\Za}{Z^\alpha}

\newcommand{\hNa}{\hat{N}^\alpha}

\newcommand{\Nan}{N^{n,\alpha}}
\newcommand{\hNan}{\hat{N}^{n,\alpha}}
\newcommand{\hN}{\hat{N}}
\newcommand{\hNn}{\hat{N}^n}
\newcommand{\Man}{M^{n,\alpha}}

\newcommand{\varrhoa}{\varrho^{\alpha}}
\newcommand{\Zn}{Z^n}
\newcommand{\Mn}{M^n}
\newcommand{\Nn}{N^n}
\newcommand{\Xn}{X^n}
\newcommand{\lambdaa}{\lambda^\alpha}
\newcommand{\Lambdaa}{\Lambda^\alpha}

\DeclareOldFontCommand{\rm}{\normalfont\rmfamily}{\mathrm}
\DeclareOldFontCommand{\sf}{\normalfont\sffamily}{\mathsf}
\DeclareOldFontCommand{\tt}{\normalfont\ttfamily}{\mathtt}
\DeclareOldFontCommand{\bf}{\normalfont\bfseries}{\mathbf}
\DeclareOldFontCommand{\it}{\normalfont\itshape}{\mathit}
\DeclareOldFontCommand{\sl}{\normalfont\slshape}{\@nomath\sl}
\DeclareOldFontCommand{\sc}{\normalfont\scshape}{\@nomath\sc}
\DeclareRobustCommand*\cal{\@fontswitch\relax\mathcal}
\DeclareRobustCommand*\mit{\@fontswitch\relax\mathnormal}

\newcommand{\diag}{\textrm{diag}}
\newcommand{\ediag}{\emph{\textrm{diag}}}
\newcommand{\Var}{\textrm{Var}}

\makeatletter
\def\newrmtheorem#1{\@ifnextchar[{\@rmothm{#1}}{\@rmnthm{#1}}}

\def\@rmnthm#1#2{%
\@ifnextchar[{\@rmxnthm{#1}{#2}}{\@rmynthm{#1}{#2}}}

\def\@rmxnthm#1#2[#3]{\expandafter\@ifdefinable\csname #1\endcsname
{\@definecounter{#1}\@addtoreset{#1}{#3}%
\expandafter\xdef\csname the#1\endcsname{\expandafter\noexpand
  \csname the#3\endcsname \@rmthmcountersep \@rmthmcounter{#1}}%
\global\@namedef{#1}{\@rmthm{#1}{#2}}\global\@namedef{end#1}{\@endrmtheorem}}}

\def\@rmynthm#1#2{\expandafter\@ifdefinable\csname #1\endcsname
{\@definecounter{#1}%
\expandafter\xdef\csname the#1\endcsname{\@rmthmcounter{#1}}%
\global\@namedef{#1}{\@rmthm{#1}{#2}}\global\@namedef{end#1}{\@endrmtheorem}}}

\def\@rmothm#1[#2]#3{\expandafter\@ifdefinable\csname #1\endcsname
  {\global\@namedef{the#1}{\@nameuse{the#2}}%
\global\@namedef{#1}{\@rmthm{#2}{#3}}%
\global\@namedef{end#1}{\@endrmtheorem}}}

\def\@rmthm#1#2{\refstepcounter
    {#1}\@ifnextchar[{\@rmythm{#1}{#2}}{\@rmxthm{#1}{#2}}}

\def\@rmxthm#1#2{\@beginrmtheorem{#2}{\csname the#1\endcsname}\ignorespaces}
\def\@rmythm#1#2[#3]{\@opargbeginrmtheorem{#2}{\csname
       the#1\endcsname}{#3}\ignorespaces}

\def\@rmthmcounter#1{\noexpand\arabic{#1}}
\def\@rmthmcountersep{}
\def\@beginrmtheorem#1#2{\rm \trivlist
      \item[\hskip \labelsep{\bf #1\ #2\thmrmcounterend}]}
\def\@opargbeginrmtheorem#1#2#3{\rm \trivlist
      \item[\hskip \labelsep{\bf #1\ #2\ (#3)\thmrmcounterend}]}
\def\@endrmtheorem{\endtrivlist}
\def\thmrmcounterend{\hskip 0em\relax}
\def\newrmwntheorem#1#2{\expandafter\@ifdefinable\csname #1\endcsname%
\global\@namedef{#1}{\@rmwnthm{#1}{#2}}%
\global\@namedef{end#1}{\@endrmwntheorem}}

\def\@rmwnthm#1#2{\@ifnextchar[{\@rmwnythm{#1}{#2}}{\@rmwnxthm{#1}{#2}}}
\def\@rmwnxthm#1#2{\@beginrmwntheorem{#2}\ignorespaces}
\def\@rmwnythm#1#2[#3]{\@opargbeginrmwntheorem{#2}{#3}\ignorespaces}
\def\@beginrmwntheorem#1{\rm \trivlist
      \item[\hskip \labelsep{\bf #1\thmrmwncounterend}]}
\def\@opargbeginrmwntheorem#1#2{\rm \trivlist
      \item[\hskip \labelsep{\bf #1\ (#2)\thmrmwncounterend}]}
\def\@endrmwntheorem{\endtrivlist}
\def\thmrmwncounterend{.\hskip 1em\relax}

\def\newsltheorem#1{\@ifnextchar[{\@slothm{#1}}{\@slnthm{#1}}}

\def\@slnthm#1#2{%
\@ifnextchar[{\@slxnthm{#1}{#2}}{\@slynthm{#1}{#2}}}

\def\@slxnthm#1#2[#3]{\expandafter\@ifdefinable\csname #1\endcsname
{\@definecounter{#1}\@addtoreset{#1}{#3}%
\expandafter\xdef\csname the#1\endcsname{\expandafter\noexpand
  \csname the#3\endcsname \@slthmcountersep \@slthmcounter{#1}}%
\global\@namedef{#1}{\@slthm{#1}{#2}}\global\@namedef{end#1}{\@endsltheorem}}}

\def\@slynthm#1#2{\expandafter\@ifdefinable\csname #1\endcsname
{\@definecounter{#1}%
\expandafter\xdef\csname the#1\endcsname{\@slthmcounter{#1}}%
\global\@namedef{#1}{\@slthm{#1}{#2}}\global\@namedef{end#1}{\@endsltheorem}}}

\def\@slothm#1[#2]#3{\expandafter\@ifdefinable\csname #1\endcsname
  {\global\@namedef{the#1}{\@nameuse{the#2}}%
\global\@namedef{#1}{\@slthm{#2}{#3}}%
\global\@namedef{end#1}{\@endsltheorem}}}

\def\@slthm#1#2{\refstepcounter
    {#1}\@ifnextchar[{\@slythm{#1}{#2}}{\@slxthm{#1}{#2}}}

\def\@slxthm#1#2{\@beginsltheorem{#2}{\csname the#1\endcsname}\ignorespaces}
\def\@slythm#1#2[#3]{\@opargbeginsltheorem{#2}{\csname
       the#1\endcsname}{#3}\ignorespaces}

\def\@slthmcounter#1{.\noexpand\arabic{#1}}
\def\@slthmcountersep{}
\def\@beginsltheorem#1#2{\sl \trivlist
      \item[\hskip \labelsep{\bf #1\ #2\thmslcounterend}]}
\def\@opargbeginsltheorem#1#2#3{\sl \trivlist
      \item[\hskip \labelsep{\bf #1\ #2\ (#3)\thmslcounterend}]}
\def\@endsltheorem{\endtrivlist}
\def\thmslcounterend{\hskip 0em\relax}
\def\newslwntheorem#1#2{\expandafter\@ifdefinable\csname #1\endcsname%
\global\@namedef{#1}{\@slwnthm{#1}{#2}}%
\global\@namedef{end#1}{\@endslwntheorem}}

\def\@slwnthm#1#2{\@ifnextchar[{\@slwnythm{#1}{#2}}{\@slwnxthm{#1}{#2}}}
\def\@slwnxthm#1#2{\@beginslwntheorem{#2}\ignorespaces}
\def\@slwnythm#1#2[#3]{\@opargbeginslwntheorem{#2}{#3}\ignorespaces}
\def\@beginslwntheorem#1{\sl \trivlist
      \item[\hskip \labelsep{\bf #1\thmslwncounterend}]}
\def\@opargbeginslwntheorem#1#2{\sl \trivlist
      \item[\hskip \labelsep{\bf #1\ (#2)\thmslwncounterend}]}
\def\@endslwntheorem{\endtrivlist}
\def\thmslwncounterend{.\hskip 1em\relax}

\makeatother

\newsltheorem{theorem}{Theorem}[section]
\newsltheorem{proposition}[theorem]{Proposition}
\newsltheorem{lemma}[theorem]{Lemma}
\newsltheorem{corollary}[theorem]{Corollary}

\newrmtheorem{remark}[theorem]{Remark}
\newrmtheorem{condition}[theorem]{Condition}
\newrmtheorem{opgave}{}[section]
\newrmtheorem{example}[theorem]{Example}
\newrmtheorem{definition}[theorem]{Definition}
\newrmtheorem{assumption}[theorem]{Assumption}

\numberwithin{equation}{section}
\let\originaleqref\eqref
\renewcommand{\eqref}{Eq.~\originaleqref}

\title{Diffusion limits for a Markov modulated binomial counting process}
\author{Peter Spreij\thanks{Korteweg-de Vries Institute for Mathematics,
Universiteit van Amsterdam and IMAPP, Radboud University Nijmegen;
{\tt spreij@uva.nl}} \and Jaap Storm\thanks{Department of Mathematics, Vrije Universiteit; {\tt p.j.storm@vu.nl}}}
\date{\today}

\begin{document}

\maketitle

\begin{abstract} 
In this paper we study limit behavior for a \emph{Markov-modulated} (MM) binomial counting process, also called a binomial counting process under \emph{regime switching}. Such a process naturally appears in the context of \emph{credit risk} when multiple obligors are present. Markov-modulation takes place when the failure/default rate of each individual obligor depends on an underlying Markov chain.
The limit behavior under consideration occurs when the number of obligors increases unboundedly, and/or by accelerating the modulating Markov process, called \emph{rapid switching}.
We establish diffusion approximations, obtained by application of (semi)martingale central limit theorems. Depending on the specific circumstances, different approximations are found.

\end{abstract}

\section{Introduction}\label{section:intro}

In this paper we study scaling limits of a \emph{Markov-modulated} (MM) counting process. Over the last decades Markov-modulation (as it is often referred to in the operations research literature) or \emph{Regime switching} (common terminology in e.g.\ mathematical finance) has become increasingly popular. Regime switching basically explains itself with its name. The parameters of the stochastic process change with time and the behavior of the process changes. The way this is usually modelled is to make the parameters of the process a function of a \emph{background process} (or \emph{modulating process}), and commonly the background process is assumed to be a finite state Markov chain, say with values in a finite set of $d$ elements. This explains the name Markov-modulation. 
The popularity of Markov-modulated processes is due to the fact that they provide a more flexible model of reality then their non-modulated versions. It is natural to assume that a real-life phenomenon, which is modelled by a stochastic process, reacts to some environment which evolves autonomously. This is far more likely than the basic case in which the parameters are constant over time. 

The process we consider has various applications. In (software) reliability modelling early variants are \citet{jelinski1972software,koch1983software,littlewood1975reliability}. The value of modelling (software) failures within randomly changing environments, including Markov-modulation, has been acknowledged for some time now, see e.g.\ \citet{ozekici2003reliability,ozekici2004reliability,ravishanker2008nhpp}.  In particular MM variants of \citet{jelinski1972software} have been studied, i.e. in a Bayersian set-up in \citet{landon2013markov}, with an estimation focus in \citet{ando2006estimating,hellmich2016statistical} and with an added failure rate component in \citet{subrahmaniam2015semiparametric}. A similar model to \citet{jelinski1972software} has been used in epidemiology (see \citet{andersson2012stochastic}) and a multivariate version of it in sampling design (see \citet{berchenko2017modeling}), where the latter can also be used to model job switching behavior due to recruiters.

An early application of Markov-modulation in economic modelling is \citet{hamilton1989new}. Since then Markov-modulation has been extensively used in various branches of mathematical finance. E.g.\ in optimal investment theory for pension funds~(\citet{chen2015optimal}), interest rate modelling~(\citet{ang2002regime,elliott2002interest,
elliott2009markov}) and affine processes~(\citet{van2014markov}). Other financial applications concern option and bond valuation~(\citet{buffington2002american,
elliott2011bond,jiang2008perpetual}), optimal dividend policies~(\citet{jiang2015optimal,
jiang2012optimal}), optimal portfolio and asset allocation~(\citet{
elliott2002portfolio,elliott1997application,zhou2003markowitz}) and also most notably in the modelling of credit risk and credit derivatives~(\citet{banerjee2016analyzing,banerjee2013pricing,
choi2015regime,
dunbar2007empirical,
giampieri2005analysis,
hainaut2016structural,
li2013pricing,liechty2013regime,
yin2009asymptotic}).
Markov-modulation has been used in insurance and risk theory as well~(\citet{asmussen2010ruin}).

Outside mathematical finance, a rich area of applications of regime switching is in operations research, where there is a sizeable body of work on Markov-modulated queues, see e.g. \citet{asmussen2008applied} and \citet{neuts1981matrix}. Contributions in this field with emphasis on scaling limits under rapid switching (leading to functional limit theorems which are also the subject of the present paper), are e.g.~\citet{anderson2016functional} and \citet{
blom2016functional}. Similar scaling limits have been obtained in~\citet{huang2016markov,
huang2014weak} and e.g.\ large deviations under scaling have been treated in~\citet{huang2016large}.

Following the considerable interest in MM financial models we consider scaling limits, also referred to as diffusion approximations, of a MM model that has a natural interpretation in Credit Risk, (see \citet{mandjes2016explicit}). In the basic setting there are $n$  obligors which have independently exponential distributed default times $\tau_i$ with intensity parameter $\lambda>0$. In the MM case this parameter is Markov-modulated, leading to an \emph{intensity process} $\lambda_t=f(Z_t)$, say for a nonnegative function $f$, where $Z$ is the Markovian background process. The process $N$ counts the number of obligors that have defaulted. At time $t$, the random variable $N_t$ is binomially distributed with parameters $n$ and $p=1-\E \exp(-\int_0^t f(Z_s) \diff s)$. Throughout the paper we will often use the credit risk context for explanation and illustration of certain features of the model, although as was explained, applications are not limited to this branch of mathematical finance.

In the present paper we study diffusion approximations (functional central limit theorems, Gaussian limits) for the process $N$, if we scale up the transition matrix of the underlying Markov chain by a factor $\alpha$ and let $n,\alpha \rightarrow \infty$. We find in principle different functional limits of the scaled and centered process, depending on the order in which parameters diverge, e.g.\ first  $\alpha \rightarrow \infty$, then  $n \rightarrow \infty$, or the other way around,  or if both $\alpha$ and $n$ jointly tend to infinity, possibly with different rates. In addition, we will also study limit behaviour for the case where the intensity vanishes at a certain rate as $n\to\infty$.

The remainder of the paper is organised as follows.  In Section~\ref{section:Z} we collect some useful results for the background Markov chain. In Section~\ref{Section Bin MM} we first construct the truly binomial process and prove in Section~\ref{sec:nonmm} a first result on diffusion approximation. Section~\ref{section:MM}, the body of the paper, is devoted to Markov modulated processes and contains the main results; we prove several limit theorems for this process in which the influence of different rates for $\alpha\to\infty$ and $n\to\infty$ is clearly visible. Some numerical examples illustrating the main results are presented in Section~\ref{section:pictures}. Finally, in Section~\ref{sec:recovery} we sketch some results for the case (in a credit risk context) where defaulted companies re-enter the market.

\section{The background process}\label{section:Z}

We will always work on a probability space \pspace. It is  assumed that the background process $Z$ is an ergodic (also called irreducible), time homogeneous Markov chain on a finite state space. Without loss of generality we assume that it takes values in the set of basis vectors $\{e_1,\ldots,e_d\}$ of $\rr^d$, with transition rates 
\begin{equation*}
q_{ji}=\frac{d}{dt}\bigg|_{t=0}\PP(Z_t=e_j|Z_0=e_i)\geq0 \quad i \neq j
\end{equation*} and $q_{ii} :=-\sum_{j \neq i} q_{ji}$. We let $Q$ be the matrix of all $q_{ij}$, also called the generator (of $Z$). Note that $\one\Top Q=0$, where $\one$ is the vector of all ones. Since $Z$ is ergodic, the limits $\pi_j={\lim_{t \rightarrow \infty}} \PP(Z_t=e_j|Z_0=e_i)$, $i,j \in \{1,\dots,d\}$ exist and are independent of $i$, and we have the column vector $\pi=(\pi_1,\dots,\pi_d)\Top$ satisfying $Q \pi=0$.
The ergodic matrix is given by $\Pi := \pi \one\Top$, has columns equal to $\pi$ and satisfies:
\begin{equation*}
\Pi^2=\Pi \textrm{ and } \Pi Q = Q\Pi = 0.
\end{equation*}
The \emph{fundamental matrix} is given by $F:=(\Pi-Q)^{-1}$. and the \emph{deviation matrix} is defined by $D:=F-\Pi$. Basic properties are: 
\begin{equation}\label{eq:D}
QF=FQ=\Pi-I, \quad F \one=\one, \quad \textrm{and } \one\Top D = 0, D\pi =0,
\end{equation}where the zeros should be read as a row or column vector. The deviation matrix can also be computed by 
\begin{equation*}
D=\int_0^\infty\exp(Qs)-\Pi \diff s,
\end{equation*} which follows from \citet[Equation~(2.14)]{glynn1984some}.

\noindent The deviation matrix of an ergodic Markov process can be interpreted as a measure of total deviation of the limiting probabilities. As such it will naturally find its way into the results of our limit theorems of the stochastic processes we observe. For a survey of the main results of deviation matrices we refer to \citet{coolen2002deviation}. 
\medskip\\
We will use a stochastic differential equation for $Z_t$. Given the process $Z_t$ on $(\Omega,\cF)$, Markovian relative to a filtration $\{\cF_t\}_{t\geq0}$, with initial state $z_0$ and with generator $Q$, one has by Dynkin's formula \citet[Proposition 1.6]{revuzyor2013} that
\begin{equation*}
\tilde M_t := Z_t - z_0 - \int_0^t Q Z_s \diff s 
\end{equation*} is a martingale relative to $\{\cF_t\}_{t\geq0}$. Rewriting this into a differential notation yields
\begin{equation}\label{dZ}
\diff Z_t=Q Z_t\diff t + \diff \tilde M_t, \quad Z_0 = z_0.
\end{equation}
This representation can be found in many papers, e.g.\ in \citet{elliott1993new}, where this result has a direct proof; see also~\citet{spreij1998representation} for a more general result. The martingale $\tilde M$ is square integrable, which implies that $\langle \tilde M \rangle_t$ exists. As a matter of fact, one has
\begin{align}\label{eq:qvarm}
\langle \tilde M \rangle_t= \int_0^t(\diag\{QZ_s\} - Q\diag\{Z_s\} - \diag\{Z_s\}Q\Top )\diff s,
\end{align}
see e.g.\ Proposition~3.2 and its proof in~\citet{huang2016markov}, and 
\begin{equation}\label{eq:dpi}
D\diag\{\pi\}+\diag\{\pi\}D\Top \mbox{ is nonnegative definite}.
\end{equation} 
Ergodicity of the Markov chain implies the \emph{continuous-time ergodic theorem} (see \citet[Theorem~3.8.1]{norris1998markov}). For $t\to\infty$, it holds that 
$\frac{1}{t} \int_0^t Z_s \diff s \overset{a.s.}{\rightarrow} \pi$. Often we will use this result in the following form,
\begin{equation}\label{eq:ergon}
\frac{1}{m} \int_0^{mt} Z_s \diff s \overset{a.s.}{\rightarrow} \pi t, \mbox{ when }m\to \infty. 
\end{equation}
\medskip\\
We close with a remark on notation. For any process $X$ we will use the generic notation $\F^X$ for the filtration generated by $X$, i.e.\ $\F^X=\{\cF^X_t\}_{t\geq 0}$, with $\cF^X_t=\sigma(X_s, 0\leq s\leq t)$.

\section{The Markov modulated binomial point process}\label{Section Bin MM}

The Markov-modulated binomial point process, or counting process, (as we refer to it) is used in a variety of applications under which are software reliability and intensity based credit risk modelling with the canonical set-up of $n$ obligors and independent default times. Especially the latter case provides a convenient context to explain some fundamental features of this process. Let us first introduce the non-modulated process.
We assume there are $n$ obligors with \emph{independent} default times $\tau^i$, $i \in \{1,\dots,n\}$. All $\tau^i$ are exponentially distributed with parameter $\lambda>0$ which gives us that the process $Y^i_t=\one_{\{t \leq \tau^i\}}$ satisfies 
\begin{equation}\label{dYi}
\diff Y^i_t = \lambda(1-Y^i_t)\diff t + \diff M^i_t, \quad Y^i_0 = 0
\end{equation}
for a martingale $M^i$ with respect to the filtration generated by $Y^i$ and the $\tau^i$. We then take $N_t:=\sum_{i=1}^n Y^i_t$ as the first process we are interested in. It then follows from the independence assumption and \eqref{dYi} that we have for the process $N$ the submartingale decomposition
\begin{equation}\label{dN Bin}
\diff N_t =\lambda(n-N_t)\diff t + \diff M_t,\quad N_0 = 0
\end{equation} 
where $M$ is an $\F^N$-martingale. We note that this model has already been introduced in Software reliability models many years ago, see e.g.~\citet{jelinski1972software,koch1983software} for early contributions and other references in Section~\ref{section:intro}. Note that for fixed $t>0$, the random variable has a Binomial$(n,p_t)$ distribution, with $p_t=1-\exp(-\lambda t)$.

This model can be generalized in many ways to one in which the (default) intensity is not a constant $\lambda$, but a time varying, random quantity $\lambda_t$. The distributional properties of $N$ are then determined by specific choices of $\lambda_t$ and equations like~\eqref{dN Bin} and its variations further on are consequences of the general martingale characterization of counting processes, see e.g.~\citet[Theorem~II.T8]{bremaud1981point}.

Our interest is to take a Markov-modulated rate $\lambda_t=\lambda\Top Z_{t-}$, where $\lambda$ is now a \emph{vector} in $\R^d_+$ (the meaning of the symbol $\lambda$ thus depends on the context, but this shouldn't cause any confusion), and $Z$ is the indicator process of the Markov chain $A$, see Section~\ref{section:Z}. We then get the following stochastic differential equation (SDE) model for the process $N$,
\begin{equation}\label{dN Bin MM}
\diff N_t = \lambda\Top Z_t(n-N_t) \diff t + \diff M_t, \quad N_0 = 0
\end{equation} where $M$ is a martingale with respect to $\F = \{\cF^Y_t \vee \cF^Z_\infty, t\geq 0\}$, which can be justified by conditional independence of the default times, given the process $Z$. In this stochastic intensity case one has that $N_t$, given the full process $Z$, has a Bin$(n,1-\exp(-\Lambda_t))$ distribution, with $\Lambda_t=\int_0^t \lambda_s \diff s = \int_0^t \lambda\Top Z_s \diff s$. We call $N$ the Markov-modulated binomial point process. See also \citet{mandjes2016explicit} for further details on the construction of this process, and for a justification of the following reasonable assumption.

\begin{assumption}\label{assumption1}
The processes $N$ and $Z$ never jump at the same time, i.e.\ the optional quadratic covariation process $[N,Z]$ is identically zero (with probability one).
\end{assumption}
There are also situations known where this assumption is violated by construction, see~\citet{spreij1990self} for an example.

\section{Limit theorems for the non-modulated binomial process}\label{sec:nonmm}

Let us first, as a warming up and for future reference, consider the truly binomial non-modulated process. Recall \eqref{dN Bin}, where we have that $\lambda>0$ is a constant. Since the process $N$ is distributed Bin$(n,1-\exp(-\lambda t))$ we have $\E N_t= n(1-\exp(-\lambda t))$. Below we will use $\varrho_t:=1-\exp(-\lambda t))$, which satisfies the ODE
\begin{equation}\label{eq:rho1}
\dot\varrho_t = \lambda(1-\varrho_t), \quad \varrho_0 = 0
\end{equation} 
This will function as the centering process for $N$, as $\E N_t=n\varrho_t$, in the following proposition.

\begin{proposition}\label{Easy CLT N Bin}
Let $\lambda>0$ be constant and let $N$ be given by \eqref{dN Bin}.  Then the scaled and centered process 
\begin{equation*}
\hat{N}^n_t := n^{-1/2}(N_t-n\varrho_t)
\end{equation*}
converges weakly to the solution of the following SDE,
\begin{equation*}
\diff \hat{N}_t=-\lambda \hat{N}_t \diff t + \diff B_t, \quad \hat{N}_0=0
\end{equation*} as $n \rightarrow \infty$.
Here $B$ is a continuous Gaussian martingale with $\langle B \rangle_t=1 - e^{ -\lambda t}$.
\end{proposition}

\begin{proof}
First we will determine the limit of the martingale $M^n=M/\sqrt{n}$ in \eqref{dN Bin}. Note that $|\Delta M^n_t|=|\Delta M_t|/\sqrt{n}  \leq 1/\sqrt{n} \rightarrow 0$, and that $\E M_t^2<\infty$ for all $t$. 
We want to prove that $\langle M^n \rangle_t \stackrel{\pp}{\rightarrow} C_t$ for some deterministic $C_t$, so that can apply the martingale central limit theorem~\citet[Theorem~VIII.3.11]{jacod2013limit}. By standard results for the compensator of a counting process,  $\langle M \rangle_t = \int_0^t \lambda(n-N_s) \diff s$. 

Using this expression for $\langle M \rangle_t$ we have that (if $n \rightarrow \infty$)
\begin{align*}
\langle M^n \rangle_t 
& = \int_0^t \lambda - \frac{\lambda N_s}{n} \diff s \\
& \overset{a.s.}{\rightarrow} \int_0^t \lambda - \lambda \E Y_t^1 \diff s \\
& = \int_0^t \lambda e^{-\lambda s} \diff s = 1 - e^{ -\lambda t},
\end{align*} where we applied the dominated convergence theorem to establish almost sure convergence of $N_t/n$ (dominated by $1$) to $\E Y^1_t$ by the strong law of large numbers. Hence we can apply the citetd martingale central limit theorem to find that $M$ converges weakly to a Gaussian martingale $B$ with $\langle B \rangle_t =1 - e^{ -\lambda t}$.

Now we consider the process $\hat{N}^n_t = n^{-1/2}( N_t - n \varrho_t)$. 
Taking the differentials and rewriting gives us
\begin{align*}
\diff \hat{N}^n_t & = - \lambda \hat{N}^n_t \diff t +  \diff M^n_t.
\end{align*}
Now we define $\hat{X}^n_t :=\exp(\lambda t) \hat{N}^n_t$, to get
$
\diff \hat{X}^n_t = e^{\lambda t} \diff M^n_t$.
By similar reasoning as in proofs of the next section where we spell out the details, we find that $\hat{X}^n$ converges in distribution to $\hat{X}=\int_0^\cdot\exp(\lambda t) \diff B_t$ and we find that $\hat{N}^n_t \overset{d}{\rightarrow} \hat{N}$, where $\hat{N}$ satisfies the SDE
\begin{equation*}
\diff \hat{N}_t = -\lambda \hat{N}_t \diff t + \diff B_t, \quad \hN_0 = 0.
\end{equation*}
\end{proof}

\begin{remark}
The binomial distribution of $N_t$ for fixed $t$, can be exploited in an application of the ordinary central limit theorem, which tells us that $\hNn_t$ has a limiting normal distribution with variance $e^{-\lambda t}(1-e^{-\lambda t})$. This is, of course, in full agreement with the functional limit result of Proposition~\ref{Easy CLT N Bin}, as can quickly be seen by computing the variance of $\hN_t$.
\end{remark}

\section{Limit results for the MM binomial process}\label{section:MM}

In this section we will prove limit results for the MM binomial point process with a Markov modulated rate. In principle, one can prove various types of limit theorems.  We focus on results in central limit form, i.e.\ on diffusion approximations. These are obtained for $n\to\infty$ in \eqref{dN Bin MM}, whereas we also investigate limit behaviours  by scaling the generator of the background process Markov chain via $Q \mapsto \alpha Q$, and letting $\alpha\to\infty$. As we are interested in the limit behaviour for both $n\to\infty$ and $\alpha\to\infty$, various possibilities occur. We will investigated iterated limits (first $n\to\infty$, then $\alpha\to\infty$ or vice versa), or joint limit behaviour when certain specified relationships between $n$ and $\alpha$ come into play. We shall also investigate the impact of different choices for the centering processes.

As a side remark, we mention that alternative scalings may lead to completely different limit results. For instance, if one would scale the vector $\lambda$ to $\frac{\lambda}{n}$, keeping $Q$ fixed, one would get a MM Poisson process, with intensity process $\lambda\Top Z_t$, see e.g.~\citet[Theorem~VIII.4.10]{jacod2013limit}, or~\citet[Theorem 1, p.~588]{liptser2012theory}. Another case, where the intensity is scaled as $\frac{\lambda}{n^\gamma}$ with $\gamma\in (0,1)$, leading again to a diffusion limit, is treated at the end of this section.

Contrary to the non-modulated case, in the MM case the consequences of a scaling $Q \mapsto \alpha Q$ for some $\alpha\to\infty$, will have a major impact. To make the dependence of the corresponding processes on $n$ and $\alpha$ explicit, we denote the resulting processes by $\Nan$, $\Man$ and $\Za$, giving the following SDE which is an analogy to \eqref{dN Bin MM}
\begin{equation}\label{dNan}
\diff \Nan_t = \lambda\Top \Za_t (n-\Nan_t) \diff t + \diff \Man_t, \quad \Nan_0 = 0.
\end{equation}
We will prove  functional limit theorems of central limit type. However the centering process $\varrho$ will, in the case $n\to\infty$, not always be the asymptotic limit of the expectation. It may depend on $\alpha$ and we will make this explicit in the notation. 
We first present a theorem for $n \rightarrow \infty$ and then $\alpha \rightarrow \infty$. Second comes the theorem in which the limits are interchanged. We write $\lambdaa_t$ for $\lambda\Top Z^\alpha_t$ and $\Lambdaa_t=\int_0^t\lambdaa_s\diff s$.

\begin{theorem}\label{MMBin hoofdstelling}
Let $\Nan$ be given by \eqref{dNan} for $\lambda \in \R^d_+$ and let $\varrhoa$ be given by \begin{equation*}
\dot\varrho^\alpha_t=\lambda\Top \Za_t (1-\varrhoa_t), \quad \varrhoa_0 = 0. 
\end{equation*} Then the scaled and compensated process 
\begin{equation*}
\hNan_t = n^{-1/2}(\Nan_t-n\varrhoa_t),
\end{equation*} converges, as $n\to\infty$, weakly to the solution of the following stochastic differential equation
\begin{equation}\label{eq:na}
\diff \hat{N}^\alpha_t = -\lambda\Top Z^\alpha_t \hat{N}^\alpha_t \diff t + \diff B^\alpha_t, \quad \hNa_0 = 0
\end{equation}
where $B^\alpha$ is a continuous martingale with $\langle B^\alpha\rangle_t=1-\exp(-\Lambda^\alpha_t)$.

Moreover, for $\alpha\to\infty$, the process $\hat{N}^\alpha$ converges weakly to the solution of 
\begin{equation}\label{eq:hn}
\diff \hN_t = -\lambdai \hN_t \diff t + \diff B_t, \quad \hN_0 = 0
\end{equation}
where $B$  is a Gaussian martingale with $\langle B \rangle_t=1-\exp(-\lambdai t)$ where $\lambdai=\lambda\Top\pi$.
\end{theorem}
\begin{proof}
We modify the proof of Proposition~\ref{Easy CLT N Bin}. We first view the scaled martingale $\Man/\sqrt{n}$, with $\Man$ defined in \eqref{dNan}. As in the proof of Proposition~\ref{Easy CLT N Bin} we see that $\Delta \Man/\sqrt{n} \rightarrow 0$. Following the same arguments, we  find that for the quadratic variation we have the expression $\langle \Man \rangle_t = \int_0^t \lambdaa_s(n-\Nan_s) \diff s$. Hence for the scaled martingale it holds that 
\begin{align*}
\langle \frac{\Man}{\sqrt{n}} \rangle_t 
& = \int_0^t \lambdaa_s (1-\frac{\Nan_s}{n})\diff s \\
& \underset{(n \rightarrow \infty)}{\overset{a.s.}{\rightarrow}} \int_0^t \lambdaa_s\exp(-\Lambdaa_s)\diff s\\
& = 1-\exp(-\Lambdaa_t), 
\end{align*} 
where we have used dominated convergence and the conditional strong law of large numbers for the convergence $\frac{\Nan_s}{n}{\overset{a.s.}{\rightarrow}}\, \E[Y^1_s|\mathcal{F}^Z]=1-\exp(-\Lambdaa_s)$. It follows from the functional CLT for martingales with random quandratic variation to a conditional Gaussian martingale \citet[Theorem 4, p.567]{liptser2012theory} that $\Man$ converges to a continuous martingale $B^\alpha$ with $\langle B^\alpha\rangle_t=1-\exp(-\Lambdaa_t)$.

One easily derives that $\hNan$ is the solution to
\[
\diff\hNan_t=-\lambdaa_t\hNan_t\diff t + n^{-1/2}\diff\Man_t,
\]
implying that 
\[
\hNan_t=n^{-1/2}\exp(-\Lambdaa_t)\int_0^t\exp(\Lambdaa_s)\diff\Man_s.
\]
It follows from the above, the validity of the P-UT condition for martingales, \citet[VI.6.13]{jacod2013limit} and the weak convergence theorem for stochastic integrals, \citet[VI.6.22]{jacod2013limit} that $\hNan$ converges to a process $\hat{N}^\alpha$, given by 
\begin{equation}\label{eq:na1}
\hat{N}^\alpha_t=\exp(-\Lambdaa_t)\int_0^t\exp(\Lambdaa_s)\diff B^\alpha_s, 
\end{equation}
which is the solution to \eqref{eq:na}.

We next consider the convergence of $\hat{N}^\alpha$ for $\alpha\to\infty$. 
From the ergodic theorem for Markov chains (see \eqref{eq:ergon}), we obtain, for $\alpha\to\infty$, $\int_0^t Z^\alpha_s\diff s \stackrel{a.s.}{\to}\pi t$, and hence $\Lambdaa_t\to\lambdai t$ and $\exp(\Lambdaa_t)\to\exp(\lambdai t)$ a.s.
As these processes are increasing and the limit is continuous we can apply \citet[Thm VI.2.15(c)]{jacod2013limit} to find that this convergence is uniform on compact sets,  
\begin{equation}\label{eq:ulam}
\underset{s\leq T}{\sup}|\exp(\Lambdaa_s)-\exp(\lambdai s)|\stackrel{\PP}{\rightarrow} 0 \quad \textrm{as } n \rightarrow \infty.
\end{equation}
Furthermore, we also obtain $\langle B^\alpha\rangle_t\stackrel{\PP}{\to} 1-\exp(-\lambdai t)$. Hence by the CLT for martingales again, we have the weak convergence of $B^\alpha$ to a continuous martingale $B$ with $\langle B\rangle_t = 1-\exp(-\lambdai t)$. 
By the same arguments as above, we have
that the stochastic integral process in \eqref{eq:na1} converges to $\int_0^\cdot\exp(\lambdai s)\diff B_s$, and therefore $\hat{N}^\alpha$ converges to the process $\hat N$ given by $\hat N_t=\exp(-\lambdai t)\int_0^t\exp(\lambdai s)\diff B_s$, which is the solution to \eqref{eq:hn}.
\end{proof}

\begin{theorem}\label{MMBin Hoofdstelling 2}
With the assumptions and notation of \autoref{MMBin hoofdstelling} we have, for $\alpha\to\infty$ that 
the counting processes $\Nan$ converge to the counting process $\Nn$ whose submartingale decomposition  is
\begin{equation*}
\diff \Nn_t = \lambdai (n- \Nn_t) \diff t + \diff M^n_t, \quad \Nn_0=0.
\end{equation*}
Equivalently, 
the centered processes $\hNan$ converge weakly to $\hNn$ defined as the solution of the SDE
\begin{equation*}
\diff \hNn_t = -\lambdai  \hNn_t \diff t + \diff \hat M^n_t, \quad \hN_0=0,
\end{equation*}
where $\hat M^n=n^{-1/2}M^n$.

Furthermore, we have that the process $\hNn$ converges weakly to $\hN$ defined as the solution of the SDE
\begin{equation*}
\diff \hN_t = -\lambdai \hN_t \diff t + \diff B_t, \quad \hN_0=0,
\end{equation*}
where $B$ is a continuous Gaussian martingale with $\langle B \rangle_t = 1-\exp(-\lambdai t)$.

\end{theorem}

\begin{proof}
The first assertion is shown in for instance the recent reference \citet{mandjes2016explicit}, Corollary~2.
For the second step we find ourselves in the situation of Proposition \ref{Easy CLT N Bin}, and if we apply this result the proof is complete.
\end{proof}

\begin{remark}
It is striking that Theorems \ref{MMBin hoofdstelling} and \ref{MMBin Hoofdstelling 2} tell that the order in which the limits are taken are taken (first $n\to\infty$, then $\alpha\to\infty$ or vice versa) give the same limit for $\hNan$. It is a priori not guaranteed that in the two situations the same limit results. Moreover, below we will investigate what happens if $\alpha$ and $n$ jointly tend to infinity, see e.g.\ Theorems~\ref{MMbin Maintheorem}, \ref{thm:mainthm2} and Proposition~\ref{prop:beta} where different limits will appear. Three different scenarios will be investigated, namely $\alpha$ tends faster to infinity than $n$, the converse situation, and the balanced case, where the speeds to convergence are proportional, and in the latter case without loss of generality equal.
\end{remark}
Up to now, we investigated limit behaviour, where  limits have been taken in specified order.
We  continue with the case where $\alpha$ and $n$ jointly tend to infinity. First we do this when this happens at the same rate for both of them, w.l.o.g.\ we take them equal, $\alpha = n$, implying the scaling $Q \mapsto nQ$. We will take the asymptotic centering process $\varrho$, similar to the one in \eqref{eq:rho1}.  We find this process by defining $\varrho_t := \lim_{n \rightarrow \infty} \frac{1}{n} \E \Nn_t= 1-\exp(-\lambdai t)$. A differential equation for $\varrho$ is given by 
\begin{equation}\label{MM Bin Asymp varrho}
\dot\varrho_t = \lambdai(1-\varrho_t), \quad \varrho_0 = 0.
\end{equation}
In this notation we have in analogy to \eqref{dN Bin MM} that the process $\Nn$ is given by
\begin{equation}\label{MM Bin dNn}
\diff \Nn_t = \lambda\Top \Zn_t(n-\Nn_t)\diff t + \diff \Mn_t, \quad \Nn_0 = 0.
\end{equation}
In the proof of Theorem~\ref{MMbin Maintheorem} the following lemma turns out to be useful, of which we shall also use a stochastic version.

\begin{lemma}\label{lemma:muf}
Consider a measurable space $(\Omega,\mathcal{F})$. Let $(\mu_n)$ be a sequence of signed measures, such that the total variations $||\mu_n||$ are bounded by a constant $B$ and that are converging weakly to a signed measure $\mu$, and let $(f_n)$ be a sequence of measurable functions, converging uniformly to $f$. Then the integrals $\mu_n(f_n)$ converge to $\mu(f)$.
\end{lemma}

\begin{proof}
Consider the inequalities
\begin{align*}
|\mu_n(f_n)-\mu(f)| 
& \leq |\mu_n(f_n)-\mu_n(f)|+|\mu_n(f)-\mu(f)|\\
& \leq ||f_n-f||_\infty ||\mu_n||+|\mu_n(f)-\mu(f)|\\
& \leq B||f_n-f||_\infty+|\mu_n(f)-\mu(f)|.
\end{align*}
By the assumptions made, both terms on the right in the above display tend to zero.
\end{proof}

\begin{remark}
Lemma~\ref{lemma:muf} also has a stochastic version. If the functions $f,f_n$ are random variables, the measures $\mu$ and $\mu_n$ are random as well (measurable in an appropriate way), the conclusion of the lemma under evidently modified conditions holds `$\omega$-wise, i.e.\ almost surely.
\end{remark}

\begin{theorem}\label{MMbin Maintheorem}
Let $\Nn$ be given by \eqref{MM Bin dNn} and $\varrho$ by \eqref{MM Bin Asymp varrho}. Then the scaled and centered process $\hNn$ given by 
\begin{equation*}
\hNn_t:=n^{-1/2}(\Nn_t-n\varrho_t),
\end{equation*} converges weakly (as $n \rightarrow \infty$) to the solution of the following SDE 
\begin{equation}\label{4.6 final SDE}
\diff \hN_t = -\lambdai \hN_t \diff t + \diff B_t + \diff G_t, \quad \hN_0 = 0,
\end{equation} where $G$ is a Gaussian martingale with 
\[
\langle G\rangle_t=\frac{V}{2\lambdai}(1-\exp(-2\lambdai t)),
\]
with $V=\lambda\Top (\ediag\{\pi\}D\Top+D\ediag\{\pi\})\lambda$ ($D$ is the deviation matrix of the background Markov chain),
and $B$ is a Gaussian martingale with $\langle B \rangle_t = 1-\exp(-\lambdai t)$, independent of $G$.
\end{theorem}

\begin{proof} We divide the proof into a number of steps.

\emph{Step 1.} We begin by rewriting the expression for $\hNn$. We have
\begin{equation}\label{eq: basis SIE}
\hNn_t = e^{-\lambda\Top \zeta^n_t} \left( \int_0^t e^{\lambda\Top \zeta^n_s} n^{1/2} \lambda\Top (\Zn_s-\pi)(1-\varrho_s)\diff s + \int_0^t e^{\lambda\Top \zeta^n_s} \diff n^{-1/2} \Mn_s \right),
\end{equation} where $\zeta^n_s = \int_0^s \Zn_u \diff u$. Now consider the process 
\begin{equation*}
\Xn_t = \int_0^t e^{\lambda\Top \zeta^n_s} n^{1/2} \lambda\Top (\Zn_s-\pi)(1-\varrho_s)\diff s + \int_0^t e^{\lambda\Top \zeta^n_s} \diff n^{-1/2} \Mn_s.
\end{equation*} Define $\Psi^n_s = \exp(\lambda\Top \zeta^n_s) (1-\varrho_s)\lambda\Top D$ and recall from \eqref{dZ} that $\Zn_t - \Zn_0 - n\int_0^t Q \Zn_s \diff s = \tilde{M}^n_t$ for a martingale $\tilde M^n$. From~\eqref{eq:D} we obtain $DQ = \Pi -I$, $\Pi\Zn_t = \pi$. Hence, we can write
\begin{equation}\label{eq: 2nd SIE}
\Xn_t = - \int_0^t \Psi^n_s \diff n^{-1/2} \Zn_s + \int_0^t \Psi^n_s \diff n^{-1/2} \tilde{M}^n_s + \int_0^t e^{\lambda\Top \zeta^n_s} \diff n^{-1/2} \Mn_s.
\end{equation}

\emph{Step 2.} 
In order to prove weak convergence of the process in \eqref{eq: basis SIE} we prove  joint weak convergence of \eqref{eq: 2nd SIE} and $e^{-\lambda\Top \zeta^n_t}$. 
By using the same reasoning as in the proof of \autoref{MMBin hoofdstelling} in order to arrive at \eqref{eq:ulam}, we have $\exp(\lambda\Top \zeta^n_t) \overset{ucp}{\rightarrow} \exp(\lambdai t)$ and the u.c.p.\ convergence of $\exp(-\lambda\Top \zeta^n_t)$ to $\exp(-\lambdai t)$.

\emph{Step 3.} In order to establish weak convergence of \eqref{eq: 2nd SIE} we establish joint weak convergence of the terms. We begin with showing that the first term converges weakly to the zero process.
Using the result from Step 2 and $1-\varrho_s=\exp(-\lambdai s)$, we get from the continuous mapping theorem that
\begin{equation}\label{eq:psiucp}
    \Psi^n \overset{ucp}{\rightarrow} \Psi := \lambda\Top D.
\end{equation}
We have that the processes $t\mapsto\exp(\lambda\Top \zeta^n_t)$ are of bounded variation on compact intervals,  uniformly in $n$. Therefore, as $t \mapsto e^t$ is Lipschitz on compact sets and $\varrho_s$ is of bounded variation, we also have that the $\Psi^n$ are of bounded variation and have bounded total variation processes on compact sets uniformly in $n$.
It follows that $n^{-1/2} \Psi^n$ converges u.c.p.\ to the zero process and so does its total variation process. To analyze the integral $\int_0^t \Psi^n_s \diff  \Zn_s$ we make the following observation, derived from~\citet{jansen2018scaling}. Every component $Z^{n,i}$ of the process $Z^n$ takes values in $\{0,1\}$, and hence $\Delta Z^{n,i}_s\in\{-1,+1\}$. Therefore the  integral $\int_0^t \Psi^n_s \diff  Z^{n,i}_s$ is of the form $\sum_{\tau_i\leq t} \pm\Psi^n_{\tau_i}$, where the $\tau_i$ are the jump times of $Z^{n,i}$. Hence $|\int_0^t \Psi^n_s \diff  Z^{n,i}_s|$  is bounded from above by the sum of the total variation of $Z^{n,i}$ and its sup-norm, see ~\citet[Lemma~6.6.5]{jansen2018scaling}.
Therefore we have for the first term of \eqref{eq: 2nd SIE}, that 
\begin{equation*}
     \int_0^t \Psi^n_s \diff n^{-1/2} \Zn_s \overset{ucp}{\rightarrow} 0.
\end{equation*} 
By Slutsky's theorem, one obtains joint convergence of the three terms in \eqref{eq: 2nd SIE}, as soon as the final two terms jointly converge weakly. This we show in the next step.

\emph{Step 4.} For the weak convergence of the remaining two terms of \eqref{eq: 2nd SIE} consider the locally square-integrable martingale 
\begin{equation*}
    \mathbf{M}^n_t=\begin{bmatrix}
    n^{-1/2}\tilde{M}^n_t \\
    n^{-1/2}\Mn_t
    \end{bmatrix}
\end{equation*} which by Assumption~\ref{assumption1} has quadratic variation given by
\begin{equation*}
    \langle\mathbf{M}^n \rangle_t
    = \int_0^t \begin{bmatrix}
    \tilde{V}^{n,*}_s & 0 \\
    0 & V^{n,*}_s 
    \end{bmatrix} \diff s,
\end{equation*} where 
\begin{align*}
    &  \tilde{V}^{n,*}_s = \diag\{Q\Zn_s\} - Q\diag\{\Zn_s\} - \diag\{\Zn_s\}Q\Top, \\
    & V^{n,*}_s = \lambda\Top \Zn_s (1 - n^{-1} \Nn_s)
\end{align*} because of \eqref{eq:qvarm}.
Therefore, by following \citet[section III.6a]{jacod2013limit}, the square integrable martingale
\begin{equation}\label{4.6 2d martingale}
    \mathbf{M}^{\zeta,n}_t=\begin{bmatrix}
    \int_0^t \Psi^n_s \diff n^{-1/2}\tilde{M}^n_s \\
    \int_0^t e^{\lambda\Top \zeta^n_s} \diff n^{-1/2} \Mn_s
    \end{bmatrix}  
\end{equation} has quadratic variation 
\begin{align}
    \langle\mathbf{M}^{\zeta,n} \rangle_t & =  \int_0^t \begin{bmatrix}
    \Psi^n_s \tilde{V}^{n,*}_s (\Psi^n_s)\Top & 0 \\
    0 & \exp(2\lambda\Top \zeta^n_s) V^{n,*}_s 
    \end{bmatrix} \diff s.\label{eq:qvm}
\end{align}
In order to prove weak convergence of the last two terms in \eqref{eq: 2nd SIE} we aim to apply the MCLT to the martingale $\mathbf{M}^{\zeta,n}$ in \eqref{4.6 2d martingale}. Thereto we need that (i) the jumps of the martingale on compact sets disappear and that (ii) the quadratic variation converges to a deterministic continuous function in probability.

For (i) the proof is that both integrals in \eqref{4.6 2d martingale} have continuous integrands. Therefore the stochastic integral with respect to $\Mn$ ($\tilde{M}^n$ can be treated in the same way) we have 
\begin{equation*}
    \max_{0 \leq t \leq T} \left\{ |\Delta \int_0^t \Psi^n_s \diff n^{-1/2} \Mn_s|\right\} \leq \lVert \Psi^n \rVert_\infty  n^{-1/2} \rightarrow 0,
\end{equation*} where $\lVert \Psi^n \rVert_\infty$ denotes the supremum-norm of $\Psi^n$ on $[0,T]$ which is finite as $\Psi^n$ is bounded on compact intervals uniformly in $n$. 
For (ii) we check the convergence of the two non-zero entries in the quadratic variation separately. 

First we consider $\int_0^t \Psi^n_s \tilde{V}^{n,*}_s (\Psi^n_s)\Top \diff s$. We know that $\Psi^n$ converges u.c.p.\ to $\Psi$, see \eqref{eq:psiucp}, and from \eqref{eq:ergon} that $\int_0^t \Zn_s \diff s$ converges a.s.\ to $\pi t$. Below we apply the almost sure version of Lemma~\ref{lemma:muf} to  the elements of the matrix $\int_0^t \Psi^n_s \tilde{V}^{n,*}_s (\Psi^n_s)\Top \diff s$. 
Take the $ij$-th element of this matrix. It is, in obvious notation, a sum over $k$ and $l$ of integrals $\int_0^t (\Psi^n_s)_{ik} (\tilde{V}^{n,*}_s)_{kl} (\Psi^n_s)_{jl} \diff s$, where those integrals are of the form $\int_0^t (\Psi^n_s)_{ik}(\Psi^n_s)_{jl} \mu_{kl}(\diff s)$, with   $\mu_{kl}(\diff s)=(\tilde{V}^{n,*}_s)_{kl} \diff s$. Using that $\int_0^t(\tilde{V}^{n,*}_s) \diff s\to (\diag\{\pi\}D\Top+D\diag\{\pi\})\, t$, we see that an application of Lemma~\ref{lemma:muf} results in\begin{align}
    & \int_0^t \Psi^n_s \left(\diag\{Q\Zn_s\} - Q\diag\{\Zn_s\} - \diag\{\Zn_s\}Q\Top\right) (\Psi^n_s)\Top \diff s \label{eq:psi2}\\
     & \overset{ucp}{\rightarrow}  - \int_0^t \Psi_s (Q\diag\{\pi\}+\diag\{\pi\}Q\Top) \Psi_s\Top \diff s \nonumber \\
     & = \lambda\Top (\diag\{\pi\}D\Top+D\diag\{\pi\})\lambda\, t=: V\,t,\nonumber
\end{align} 
since $DQ = \Pi -I$ and $Q\pi = 0$, see \eqref{eq:D}. Note that $V$ is nonnegative in view of \eqref{eq:dpi}.

Next we consider the other non-zero entry in the quadratic variation. Recall $V^{n,*}_t= \lambda\Top \Zn_t (1 - n^{-1} \Nn_t)$. We first show u.c.p.\ convergence of $\int_0^t  V^{n,*}_s \diff s$ to a continuous function. This requires a couple of steps.

The first step is to show that $n^{-1}\Nn_t$ converges in probability to $\varrho_t=1-\exp(\lambdai t)$, as a matter of fact we show that the convergence is in the $L^2$-sense. Recall that, conditional on  $\cF^Z$, $N^n_t$ is binomial with parameters $n$ and $p^n_t=1-\exp(-\int_0^t \lambda\Top\Zn_s\diff s)$. Therefore we have $\E(n^{-1}\Nn_t-\varrho_t) = \E p^n_t-\varrho_t\to 0$. Hence, writing $\E (n^{-1}\Nn_t-\varrho_t)^2=(\E(n^{-1}\Nn_t-\varrho_t))^2+ \Var(n^{-1}\Nn_t)$, we only have to prove that the variance tends to zero.
By the law of total variation we have
\begin{align*}
    \Var(n^{-1}\Nn_t) & = n^{-2}\E\,\Var(N^n_t | \cF^Z) +\Var(\E[n^{-1}N^n_t | \cF^Z])\\ 
    & = n^{-1} \E[p^n_1(1-p^n_t)]+\Var\, p^n_t.
\end{align*} 
As the $p^n_t$ are bounded and $p^n_t \to 1-\exp(-\lambda_\infty t)$, we get $\Var(n^{-1}\Nn_t)\to 0$ by application of the dominated convergence theorem.

Now we are ready for the final step. Write
\begin{align}\label{eq:v*}
\lefteqn{\int_0^t \exp(2\lambda\Top \zeta^n_s) V^{n,*}_s \diff s} \nonumber\\
& = \int_0^t \exp(2\lambda\Top \zeta^n_s) \lambda\Top \Zn_s (1 - \varrho_s) \diff s + \int_0^t \exp(2\lambda\Top \zeta^n_s) \lambda\Top \Zn_s (\varrho_s - n^{-1} \Nn_s) \diff s.
\end{align}
Consider the expectation of the absolute value of the last integral. By the Cauchy-Schwartz inequality, its square is less than
\[
\E\int_0^t (\exp(2\lambda\Top \zeta^n_s) \lambda\Top \Zn_s)^2\diff s\times \int_0^t \E(\varrho_s - n^{-1} \Nn_s)^2 \diff s. 
\]
In this product, the first factor is bounded, whereas the second factor tends to zero by the above $L^2$-convergence of $n^{-1}N^n_s$ to $\varrho_s$ and by application of the monotone convergence theorem.

We now focus on the first term on the RHS of~\eqref{eq:v*}. By the ucp-convergence of the exponential term $\exp(2\lambda\Top \zeta^n_s) (1 - \varrho_s)$ to $\exp(2\lambdai s) (1 - \varrho_s)=\exp(\lambdai s)$ (as in Step 2), we can again apply the almost sure version of Lemma~\ref{lemma:muf}, to arrive at $\int_0^t \exp(2\lambda\Top \zeta^n_s) \lambda\Top \Zn_s (1 - \varrho_s) \diff s \stackrel{a.s.}{\to} \int_0^t \exp(\lambdai s) \lambdai\diff s=\exp(\lambdai t)-1$.
Summarizing all these intermediate results we get convergence in probability of the quadratic variation, i.e. \begin{equation*}
    \langle \mathbf{M}^{\zeta,n} \rangle_t \overset{\PP}{\rightarrow} \begin{bmatrix}
    Vt & 0 \\
    0 & e^{\lambdai t}-1
    \end{bmatrix}.
\end{equation*}
The MCLT allows us to deduce that $\mathbf{M}^{\zeta,n}$ converges weakly to a two-dimensional Gaussian martingale with independence of the components.

\emph{Step 5:} By the weak convergence of $\mathbf{M}^{\zeta,n}$ and weak convergence of the first term in \eqref{eq: 2nd SIE} to the zero process, we deduce by application of the continuous mapping theorem, weak convergence of \eqref{eq: 2nd SIE} to 
a Gaussian martingale with quadratic variation
$\int_0^t (V + \lambdai e^{\lambdai s}) \diff s$.
Therefore $\mathbf{M}^{\zeta,n}$ has the limit distribution of 
\begin{equation}\label{4.6 eq4} 
    \begin{bmatrix}
    \int_0^t\sqrt{V}\diff B^1_s \\
    \int_0^t\sqrt{\lambdai e^{\lambdai s}} \diff B^2_s
    \end{bmatrix},
\end{equation} where $B^1$ and $B^2$ are independent standard Brownian motions and thus we have weak convergence of the sum in \eqref{eq: 2nd SIE} to the Gaussian martingale in \eqref{4.6 eq4}.

\emph{Step 6:} In conclusion we showed that $\hNan$ can be written as the product of processes in \eqref{eq: basis SIE}. In Step 2 we show u.c.p.\ convergence of the first process and in Steps 3--5 we showed weak convergence of the second process. Therefore we have joint weak convergence of the processes in 
\eqref{eq: basis SIE}. Since multiplication is continuous at continuous limits in the Skorohod topology (c.f. \citet[Thm 4.2]{whitt1980some}) we have weak convergence of $\hNn$ to the process $\hat N$ given by ($\tilde B$ is a standard Brownian motion)
\begin{equation*}
    \hat{N}_t = e^{-\lambdai t} \int_0^t \sqrt{V  + \lambdai e^{\lambdai s}} \diff \tilde B_s,
\end{equation*} 
which solves the SDE
\[
\diff\hat N_t=-\lambdai \hat N_t\diff t + e^{-\lambdai t}\sqrt{V  + \lambdai e^{\lambdai t}} \diff \tilde B_t.
\]
Alternatively, we can represent this SDE as \eqref{4.6 final SDE},
\[
\diff\hat N_t=-\lambdai \hat N_t\diff t + \diff B_t + \diff G_t,
\]
where $B$ and $G$ are independent Gaussian martingales, with $\langle B\rangle_t=1-\exp(-\lambdai t)$ and $\langle G\rangle_t=\frac{V}{2\lambdai}(1-\exp(-2\lambdai t))$.
\end{proof}

\begin{remark}
Let's compare Proposition~\ref{Easy CLT N Bin} and Theorem~\ref{MMbin Maintheorem}. The Brownian motion $B$ of Theorem~\ref{MMbin Maintheorem} is as $B$ in Proposition~\ref{Easy CLT N Bin}. The Brownian motion $G$ in the theorem comes as an ``extra'' compared to the situation of the proposition. If we apply Theorem~\ref{MMbin Maintheorem} to the non-modulated case, which happens if the vector $\lambda$ is a constant $\lambdai$ times $\one$, we have  $\langle G \rangle_t = -\lambdai^2 \one\Top (\diag\{\pi\}D\Top+D\diag\{\pi\})\one\, t$, which is indeed zero in view of the property $D\pi=0$, see \eqref{eq:D}. So this theorem in the non-modulated case reduces to Proposition~\ref{Easy CLT N Bin}, as it should.
\end{remark}

The centering in Theorem~\ref{MMbin Maintheorem} is with the function $n\varrho_t$. In the proposition below we compare  $\varrho_t=1-\exp(-\lambdai t)$ with $\varrho^n_t=1-\exp(-\lambda\Top\zeta^n_t)$, and we shall see  the result of alternative centering  with $\varrho^n_t$ in Proposition~\ref{prop:p2}.

\begin{proposition}\label{prop:p1}
It holds that $\hat{H}^n_t:=\sqrt{n}(\varrho^n_t-\varrho_t)$ converges weakly to the process $\hat H$ given by $\hat{H}_t=\exp(-\lambdai t)G^H_t$,
where $G^H$ is a 
Brownian motion with variance parameter $V=\lambda\Top (\diag\{\pi\}D\Top+D\diag\{\pi\})\lambda$, so $\langle G^H\rangle_t=Vt$. 
\end{proposition}

\begin{proof}
We compute
\begin{align*}
\varrho^n_t-\varrho_t & = \exp(-\lambdai t)-\exp(-\lambda\Top\zeta^n_t) \\
& =\int_0^t \exp(-\lambda\Top\zeta^n_s) \lambda\Top\zeta^n_s\diff s-\int_0^t \exp(-\lambdai s)\lambdai\diff s \\
& =\int_0^t (\exp(-\lambda\Top\zeta^n_s)-\exp(-\lambdai s)) \lambda\Top Z^n_s\diff s+\int_0^t \exp(-\lambdai s)\lambda\Top(Z^n_s-\pi)\diff s \\
& =-\int_0^t (\varrho^n_s-\varrho_s) \lambda\Top Z^n_s\diff s+\int_0^t \exp(-\lambdai s)\lambda\Top(Z^n_s-\pi)\diff s.
\end{align*}
For $\hat{H}^n_t=\sqrt{n}(\varrho^n_t-\varrho_t)$ we then obtain
\[
\hat{H}^n_t=-\int_0^t \hat{H}^n_s\lambda\Top Z^n_s\diff s + \sqrt{n}\int_0^t \exp(-\lambdai s)\lambda\Top(Z^n_s-\pi)\diff s.
\]
Solving this equation yields
\[
\hat{H}^n_t=\exp(-\lambda\Top \zeta^n_t)\sqrt{n}\int_0^t \exp(\lambda\Top \zeta^n_s-\lambdai s)\lambda\Top(Z^n_s-\pi)\diff s.
\]
The latter equation has the same structure as \eqref{eq: basis SIE}, but with the martingale term missing. For the limit behaviour we can therefore copy the relevant parts of the proof of Theorem~\ref{MMbin Maintheorem}, which yields the assertion.
\end{proof}
Now we revisit Theorem~\ref{MMbin Maintheorem}, by replacing the centering $n\varrho_t$ by $n\varrho^n_t$. This leads to

\begin{proposition}\label{prop:p2}
Let $\hat{K}^n_t=n^{-1/2}(N^n_t-n\varrho^n_t)$. Then $\hat{K}^n$ converges weakly to the solution to the SDE $\diff \hat{K}_t=-\lambdai \hat{K}_t\diff t + \diff B_t$,
where $B$ is a continuous Gaussian martingale with $\langle B\rangle_t=1-\exp(-\lambdai t)$.
\end{proposition}

\begin{proof}
We follow the line of reasoning of the proof of Theorem~\ref{MMbin Maintheorem}. Parallel to the first step of that proof we now obtain
\begin{equation}\label{eq:k}
\hat{K}^n_t=n^{-1/2}M^n_t-\int_0^t \lambda\Top Z^n_s \hat{K}^n_s\diff s,
\end{equation}
equivalent to
\[
\hat{K}^n_t=\exp(-\lambda\Top \zeta^n_t)\,\int_0^t n^{-1/2}\exp(\lambda\Top \zeta^n_s)\diff M^n_s,
\]
which is a simpler version of \eqref{eq: basis SIE}. Following the steps in the proof of the theorem, we arrive at the weak convergence of the stochastic integral to a Gaussian martingale $\tilde B$ with quadratic variation $\exp(\lambdai t)-1$ and of the process $\hat{K}^n$ to $\hat{K}$ given by $\hat{K}_t=\exp(-\lambdai t)\tilde B_t$. The latter being equivalent to $\hat{K}$ solving
\[
\diff \hat{K}_t=-\lambdai \hat{K}_t\diff t + \diff B_t,
\]
where $B$ is a Gaussian martingale with $\langle B\rangle_t=1-\exp(-\lambdai t)$.
\end{proof}
Putting the conclusions of Propositions~\ref{prop:p1} and~\ref{prop:p2} together and comparing them to Theorem~\ref{MMbin Maintheorem}, we can provide an illuminating explanation of the result of the theorem. Write $\hat N^n_t=\hat H^n_t+\hat K^n_t$,  and recall the following weak limits. We have seen that the limit process $\hat H$ satisfies $\diff \hat H_t=-\lambdai \hat H_t\diff t + \exp(-\lambdai t)\diff G^H_t$, with $\langle G^H\rangle_t=Vt$. And we have also seen that the limit process $\hat K$ is a Gaussian process satisfying $\diff \hat K_t=-\lambdai \hat K_t\diff t + \diff B_t$, with $\langle B\rangle_t=1-\exp(-\lambdai t)$. Adding up these limits (justified by the proof of the theorem) yields that the limit process $\hat N$ satisfies $\diff\hat N_t=-\lambdai \hat N_t\diff t + \diff B_t+ \exp(-\lambdai t)\diff G^H_t$. With $G_t=\int_0^t \exp(-\lambdai s)\diff G^H_s$, we have $\langle G\rangle_t$ as in the theorem, and the SDE \eqref{4.6 final SDE} follows again. 

Summarizing, the result of Theorem~\ref{MMbin Maintheorem} can be explained by decoupling $\hat N^n$ into $\hat H^n$ and $\hat K^n$ and their limits results according to Propositions~\ref{prop:p1} and~\ref{prop:p2}. From a distributional point of view, the result of Theorem~\ref{MMbin Maintheorem} is more appealing than Proposition~\ref{prop:p2}, since the latter involves centering with a random process. In line with the common view on a central limit theorem, one can interpret the statement of the theorem by saying that asymptotically, for fixed $t$, the random variable $\hat N^n_t$ has a normal distribution with (nonrandom) mean $n\varrho_t$ and variance $n\exp(-2\lambdai t)\left(Vt+\exp(\lambdai t)-1\right)$.

Note that the $\hat K_t$ in Proposition~\ref{prop:p2} is the same limiting process as the limiting process $\hat N_t$ in \autoref{MMBin hoofdstelling} and \autoref{MMBin Hoofdstelling 2}. In continuation of our discussion above we can explain the equivalence of these limits via the centering process. The process $\hat{K}^n$ is centered with the stochastic process $n\varrho^n$ (and likewise we use for $\hat N^{n,\alpha}$ centering with the stochastic process $n\varrho^\alpha$). Centering in this way, as opposed to centering with $n\varrho_t$, removes the first term in \eqref{eq: basis SIE}, which in \autoref{MMbin Maintheorem} converges to the Gaussian term $G$ with $\langle G\rangle_t=\frac{V}{2\lambdai}(1-\exp(-2\lambdai t))$. Intuitively one cancels out the 'extra' variation, due to the first term \eqref{eq: basis SIE} which results in a situation where the order of limits does not matter anymore. This situation is to some extent similar to the case for the process given by \eqref{eq:na}. But note also the difference between the two cases, the martingale in~\eqref{eq:na} is continuous and Gaussian, whereas the martingale in \eqref{eq:k} is a (scaled) compensated jump martingale, although with a continuous Gaussian limit.
\medskip\\
Next we investigate the limit behaviour of $N_t$ when we speed up the Markov chain with $n^\beta$ for some $\beta>0$. Note that before we had $\beta=1$. The Propositions~\ref{prop:p1} and~\ref{prop:p2} now take a different form, but the proofs of the results in Proposition~\ref{prop:p3} below are similar to the previous ones, and are therefore omitted. Let us write, in order to express the dependence on $\beta$, $\varrho^{n,\beta}=1-\exp(-\lambda\Top\zeta^{n,\beta}_t)$ with $\zeta^{n,\beta}_t=\int_0^t Z_{n^\beta s}\diff s$.

\begin{proposition}\label{prop:p3}
(i) Let $\hat H^{n,\beta}_t:=n^{\beta/2}(\varrho^{n,\beta}_t-\varrho_t)$. Then $\hat H^{n,\beta}$  converges weakly to the process $\hat H$ given by $\hat H_t=\exp(-\lambdai t)G^H_t$,
where $G^H$, as before, is a 
Brownian motion with variance parameter  $V=\lambda\Top (\diag\{\pi\}D\Top+D\diag\{\pi\})\lambda$.\\
(ii) Let $\hat K^{n,\beta}_t=n^{-1/2}(N^n_t-n\varrho^{n,\beta}_t)$. Then $\hat K^{n,\beta}$ converges weakly to the solution to the SDE $\diff \hat K_t=-\lambdai \hat K_t\diff t + \diff B_t$,
where $B$ is a continuous Gaussian martingale with $\langle B\rangle_t=1-\exp(-\lambdai t)$.
\end{proposition} 
As a consequence of this proposition, we have the following extension of Theorem~\ref{MMbin Maintheorem}.

\begin{theorem}\label{thm:mainthm2}
Let $\Nn$ be given by \eqref{MM Bin dNn} and $\varrho$ by \eqref{MM Bin Asymp varrho}. Then the scaled and centered process $\hat{N}^{n,\beta}$ given by 
\begin{equation*}
\hat{N}^{n,\beta}_t :=n^{-1/2(1+(1-\beta)^+)}(N^n_t-n\varrho_t),
\end{equation*} converges weakly (as $n \rightarrow \infty$) to the solution of the following SDE 
\begin{equation}
\diff \hN_t = -\lambdai \hN_t \diff t  +\one_{\{\beta\leq 1\}} \diff G_t+ \one_{\{\beta\geq 1\}} \diff B_t, \quad \hN_0 = 0,
\end{equation} where $G$  is a Gaussian martingale with 
\[
\langle G\rangle_t=\frac{V}{2\lambdai}(1-\exp(-2\lambdai t)),
\]
with $V=\lambda\Top (\ediag\{\pi\}D\Top+D\ediag\{\pi\})\lambda$ ($D$ is the deviation matrix of the background Markov chain),
and $B$  is a Gaussian martingale with $\langle B \rangle_t = 1-\exp(-\lambdai t)$, independent of $G$. 

Alternatively, we have the representation
\[
\diff \hN_t = -\lambdai \hN_t \diff t  +
\left(\one_{\{\beta\leq 1\}} V\exp(-2\lambdai t) +
\one_{\{\beta\geq 1\}}\lambdai \exp(-\lambdai t)\right)^{1/2} \diff W_t,
\]
where $W$ is a standard Brownian motion.
\end{theorem}

\begin{proof}
We only have to consider the cases $\beta<1$ and $\beta>1$, as the case $\beta=1$ is covered by Theorem~\ref{MMbin Maintheorem}. For $\beta<1$ we have $\hat{N}^{n,\beta}=n^{-1+\beta/2}(\Nn_t-n\varrho_t) = n^{(\beta-1)/2}\hat K^{n,\beta}_t+\hat H^{n,\beta}$. From Proposition~\ref{prop:p3} we obtain that $\hat{N}^{n,\beta}$ has $\hat H$ as the limit process.
For $\beta>1$ we have $\hat{N}^{n,\beta}=n^{-1/2}(\Nn_t-n\varrho_t) = \hat K^{n,\beta}_t+n^{(1-\beta)/2}\hat H^{n,\beta}$. From Proposition~\ref{prop:p3} we now obtain that $\hat{N}^{n,\beta}$ has $\hat K$ as the limit process.

Putting these two cases (combined with $\beta=1$) together we see that we obtain for $\hat{N}^{n,\beta}$ the limit process $\one_{\{\beta\leq 1\}}\hat H+\one_{\{\beta\geq 1\}}\hat K$. Therefore we also have
\begin{align*}
\diff \hat{N}^{n,\beta}_t & =-\lambdai \hat{N}^{n,\beta}_t\diff t +\one_{\{\beta\leq 1\}} \exp(-\lambdai t)\diff G^H_t+ \one_{\{\beta\geq 1\}} \diff B_t\\
& =-\lambdai \hat{N}^{n,\beta}_t\diff t +\one_{\{\beta\leq 1\}} \diff G_t+ \one_{\{\beta\geq 1\}} \diff B_t,
\end{align*}
which completes the proof. 
\end{proof}

\begin{remark}\label{remark: Thm with beta}
From the quadratic variation of the Brownian terms in the limit of \autoref{thm:mainthm2} one sees that the quadratic variation process converges to $\frac{V}{2\lambdai}\one_{\{\beta\leq 1\}}+\one_{\{\beta\geq 1\}}$ if $t \to \infty$. Note that then also the quadratic variation $\langle \hN\rangle_t\to \frac{V}{2\lambdai}\one_{\{\beta\leq 1\}}+\one_{\{\beta\geq 1\}}$, as it is completely determined by the quadratic variation of the martingale part of $\hN$.
Therefore, the Brownian terms converge to a Gaussian random variable with expectation zero for $t\to\infty$ and variance $\frac{V}{2\lambdai}\one_{\{\beta\leq 1\}}+\one_{\{\beta\geq 1\}}$. It follows that the limiting process $\hat N$ in \autoref{thm:mainthm2} is Gaussian with vanishing variance for $t\to\infty$, and therefore behaves as the constant zero process for large $t$. 
\end{remark}

The previous limit theorems were based on a fixed value of $\lambda$, which is possibly something  to relax. Consider a modulated process, but with default intensity $n^{-\gamma} \lambda^\top Z_t$, for some $\gamma >0$. In financial terms, we consider a market with many obligors whose individual default rate tends to zero and we are still interested in the total number of defaults. A before, we scale $Q \mapsto n^\beta Q$ for some $\beta > 0$, which speeds up the background process. We omit the dependence on $\beta$ in the notation as it will turn out that the diffusion limit we derive is independent of $\beta$. So, we consider
\[
\diff N^{n,\gamma}_t=n^{-\gamma} \lambda^\top \Zn_t(n-N^{n,\gamma}_t)\diff t + \diff M^{n,\gamma}_t, \quad N^{n,\gamma}_0=0. 
\]
If $\gamma=1$, we know that the limit process (for $n\to\infty$) is a Poisson process with intensity $\lambdai$. The case $\gamma>1$ is not very interesting, one certainly has $N^{n,\gamma}_t\stackrel{a.s.}{\to} 0$ for fixed $t$. But for $0<\gamma< 1$ there is something to do; the case $\gamma=0$, we have already encountered in Proposition~\ref{Easy CLT N Bin}. Write $\lambda^{n,\gamma}_t$ for $n^{-\gamma} \lambda^\top \Zn_t$ and $\Lambda^{n,\gamma}_t =\int_0^t \lambda^{n,\gamma}_s \diff s$. Put $\varrho^{n,\gamma}_t=1-\exp(-\Lambda^{n,\gamma}_t)$.

\begin{proposition}\label{prop:beta}
Let $\beta > 0$. Consider for $\gamma\in (0,1)$ the process $\hN^{n,\gamma}$ given by 
\[
\hN^{n,\gamma}_t=n^{(\gamma-1)/2}\left(N^{n,\gamma}_t-n\varrho^{n,\gamma}_t\right). 
\]
As $n \to \infty$, this process converges weakly to a Brownian motion with variance parameter $\lambdai$.
\end{proposition}

\begin{proof}
One checks that 
\[
\diff \hN^{n,\gamma}_t=-\lambda^{n,\gamma}_t \hN^{n,\gamma}_t\diff t + \diff \hat M^{n,\gamma}_t,
\]
where $\hat M^{n,\gamma}_t= n^{(\gamma-1)/2} M^{n,\gamma}_t$, with
\[
\langle\hat M^{n,\gamma}\rangle_t =\int_0^t\lambda^\top \Zn_s (1-\frac{N^{n,\gamma}_s}{n})\diff s.
\]
Note that the jumps of $\hat M^{n,\gamma}$ disappear for $n\to\infty$, as $\gamma<1$, and that
\[
 (1-\frac{N^{n,\gamma}_t}{n})\int_0^t\lambda^\top \Zn_s\diff s\leq \langle\hat M^{n,\gamma}\rangle_t \leq \int_0^t\lambda^\top \Zn_s \diff s.
\]
As for each $t \geq 0$, $\E\frac{N^{n,\gamma}_t}{n}\to 0$ and $N^{n,\gamma}_t \geq 0$ it holds that $\frac{N^{n,\gamma}_t}{n} \overset{\PP}{\to} 0 $ and thus $\langle\hat M^{n,\gamma}\rangle_t\stackrel{\PP}{\to}\lambdai t$. Consequently, $\hat M^{n,\gamma}$ weakly converges to $\sqrt{\lambdai}B$, where $B$ is a standard Brownian motion.
As $\hN^{n,\gamma}_t=\exp(-\Lambda^{n,\gamma}_t)\int_0^t \exp(\Lambda^{n,\gamma}_s)\diff\hat M^{n,\gamma}_s$,  by previous arguments and using that $\Lambda^{n,\gamma}_t\stackrel{\PP}{\to} 0$ since $\gamma>0$, $\hN^{n,\gamma}$ converges to $\sqrt{\lambdai}B$ as well.
\end{proof}

\section{Some illustrating simulations}\label{section:pictures}

In this section we will show some graphs of simulations, illustrating some of the results proven in this paper. To illustrate all results would require too much space, so we will show two intuitive results, namely the first part of \autoref{MMBin Hoofdstelling 2} where we only speed up the underlying Markov process, and \autoref{MMbin Maintheorem}.

We simulate $N^{n,\alpha}_t$ as in \eqref{dNan} and $\hat{N}^{n,\beta}$ as in \autoref{thm:mainthm2} for a couple of parameter settings of $\alpha$ and $n$ on a time interval $[0,T]$. We take $T=3$ for the first and $T=10$ for the second simulation, to illustrate the interesting phenomona corresponding to the theorems. We take a state space of three elements for the Markov chain
\begin{equation*}
	Q = \begin{bmatrix}
		-5 & 1 & 5 \\
		 2 & -2 &  5\\
		 3 & 1 & -10
	\end{bmatrix}, 
\end{equation*} 
and the different values of the intensity are summarized by the vector $\lambda = 
	\begin{bmatrix}
		0.1 & 1 & 3
	\end{bmatrix}\Top$, a fixed choice  in all simulations.

We start by simulating $N^{n,\alpha}_t$ and $\lambda\Top Z_t$ for $n=1000$ fixed and varying values of $\alpha \in \{1,10,100,10000\}$ to illustrate the first part of \autoref{MMBin Hoofdstelling 2}. The sample paths of these simulations are shown in \autoref{fig: MMbin 1} and \autoref{fig: MMbin 2}.

\begin{figure}[H]
\centering
\begin{subfigure}{.5\textwidth}
  \centering
    \includegraphics[width=1.1\linewidth]{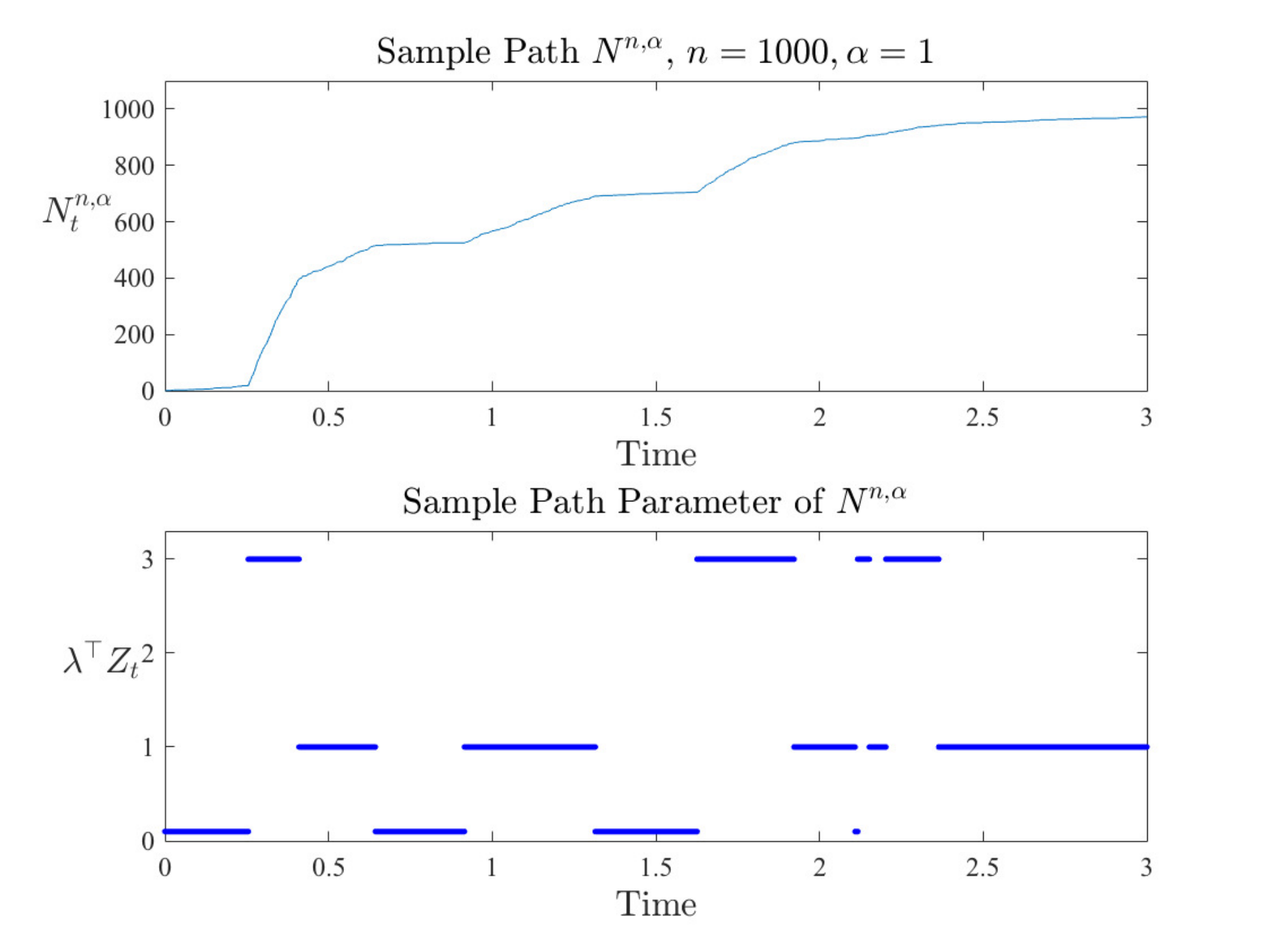} 
\end{subfigure}%
\begin{subfigure}{.5\textwidth}
  \centering
    \includegraphics[width=1.1\linewidth]{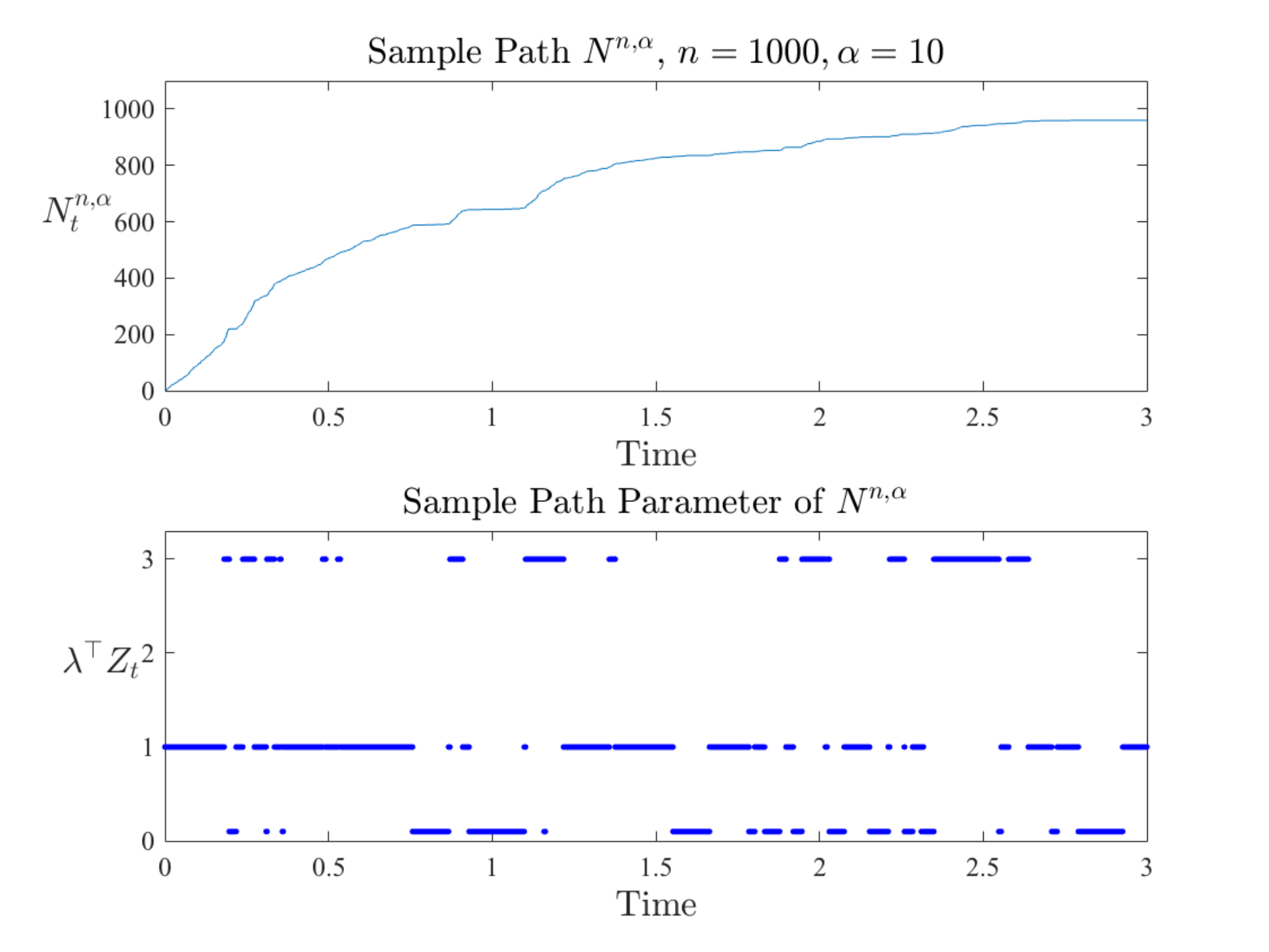} 
\end{subfigure}
\caption{Sample Paths with $\alpha=1$ (left) and $\alpha=10$ (right)}\vspace{-0.50cm}
\label{fig: MMbin 1}
\end{figure}

\begin{figure}[H]
\centering
\begin{subfigure}{.5\textwidth}
  \centering
  \includegraphics[width=1.1\linewidth]{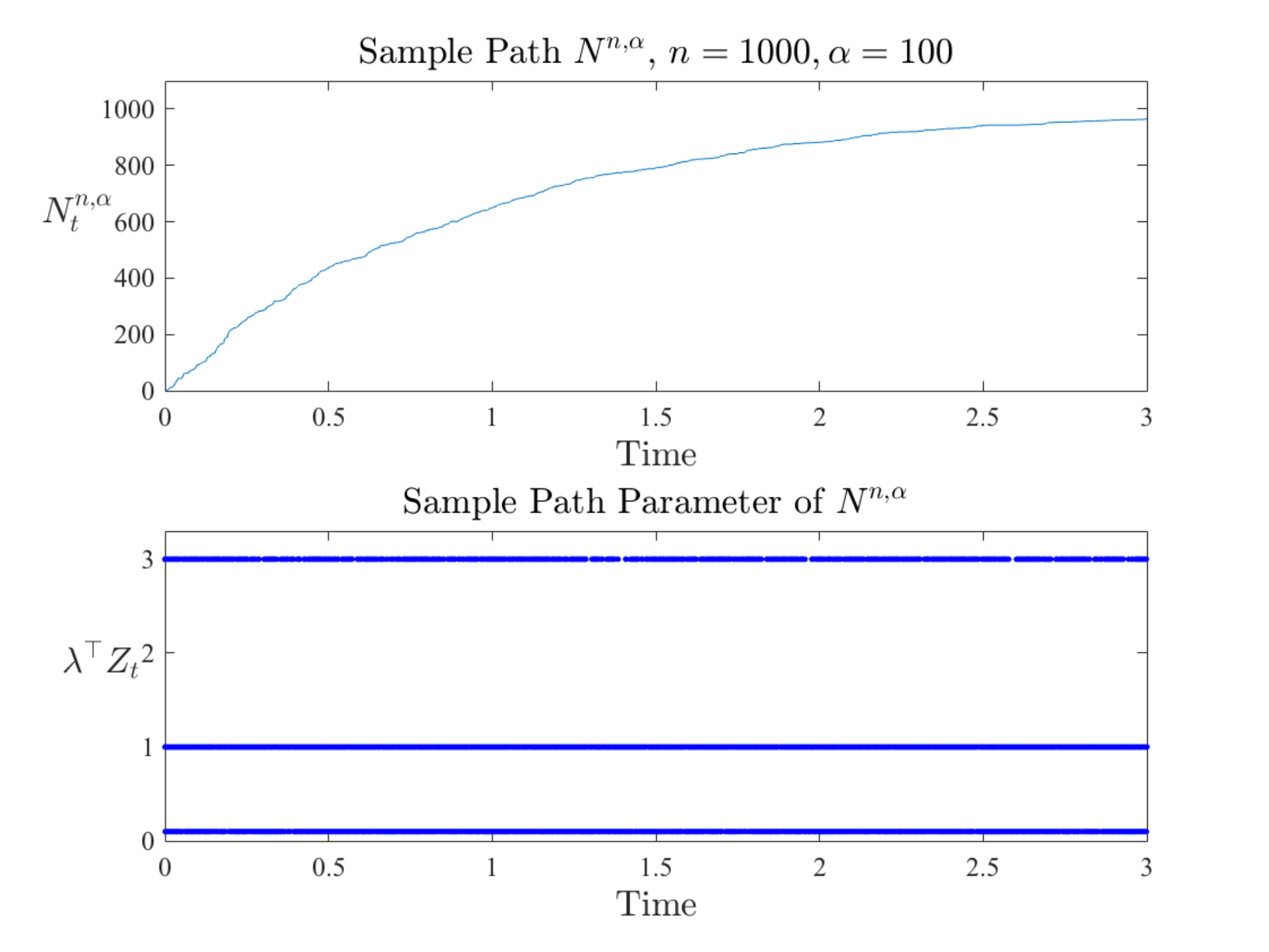} 
\end{subfigure}%
\begin{subfigure}{.5\textwidth}
  \centering
   \includegraphics[width=1.1\linewidth]{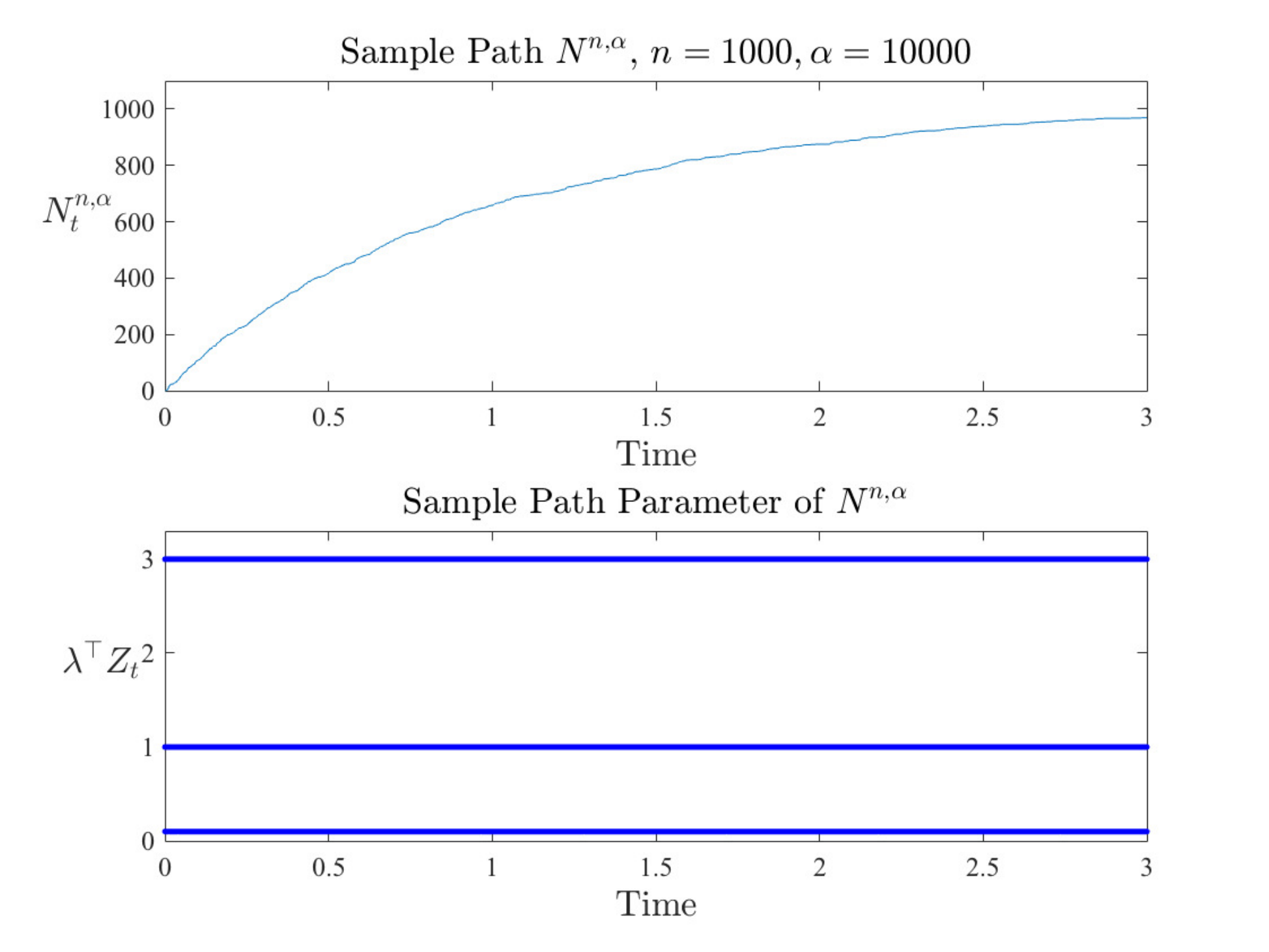} 
\end{subfigure}
\caption{Sample Paths with $\alpha=100$ (left) and $\alpha=10000$ (right)}
\label{fig: MMbin 2}
\end{figure}
One can see the effect from the Markov-modulated default rate in \autoref{fig: MMbin 1}. The contents of the first part of \autoref{MMBin Hoofdstelling 2} is that this modulating effect should disappear and the default rate becomes a deterministic constant $\lambdai$ in the limit. This is visible in \autoref{fig: MMbin 2}, where this modulating effect disappears and a constant default rate appears due to the Markov chain jumping very fast. 

Next we simulate the centered and scaled process $\hat{N}^{n,\beta}$, for $\beta =1$. We then have $\alpha=n$ and we choose $n \in \{10,100,1000,10000\}$ in order to illustrate \autoref{thm:mainthm2}. The sample paths are shown in \autoref{fig: CLT 1} and \autoref{fig: CLT 2}.

\begin{figure}[H]
\centering
\begin{subfigure}{.5\textwidth}
  \centering
\includegraphics[width=1\linewidth]{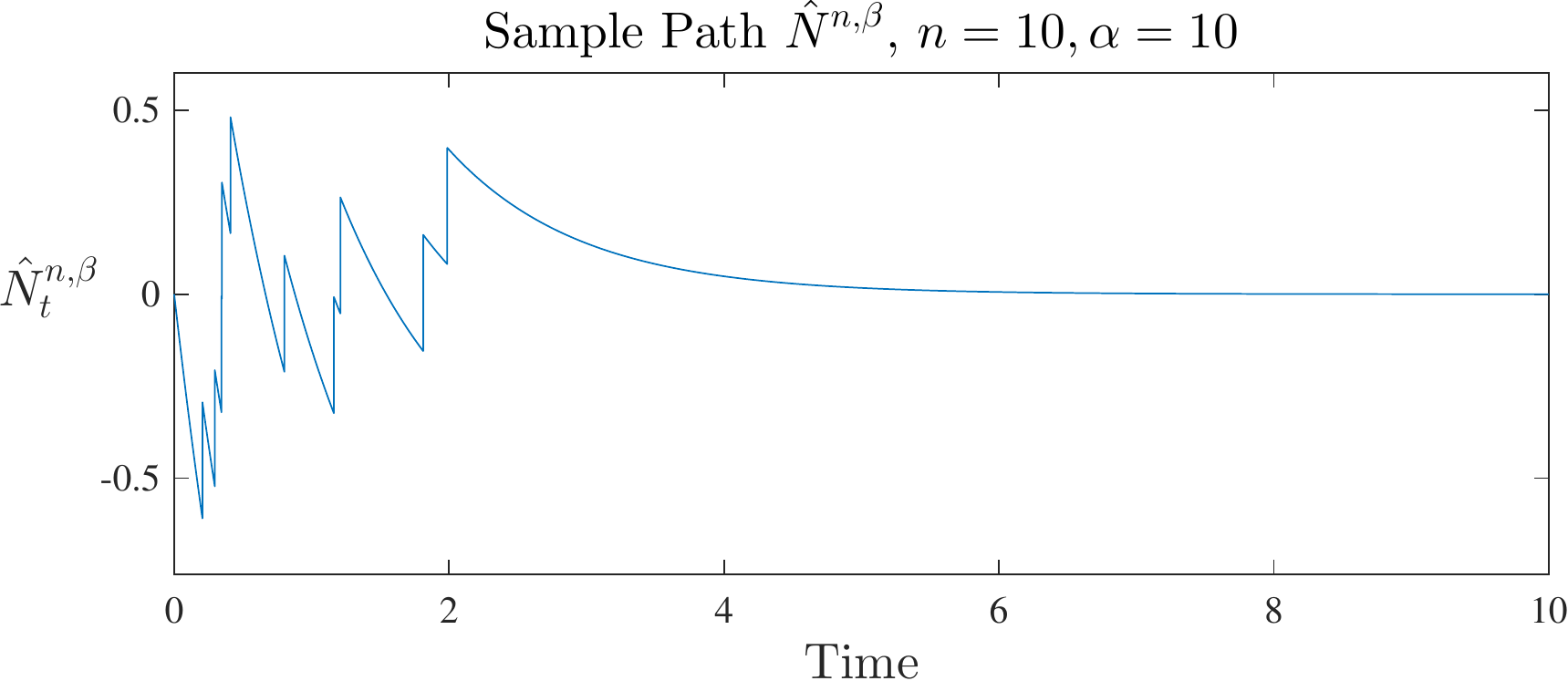} 
\end{subfigure}%
\begin{subfigure}{.5\textwidth}
  \centering
    \includegraphics[width=1\linewidth]{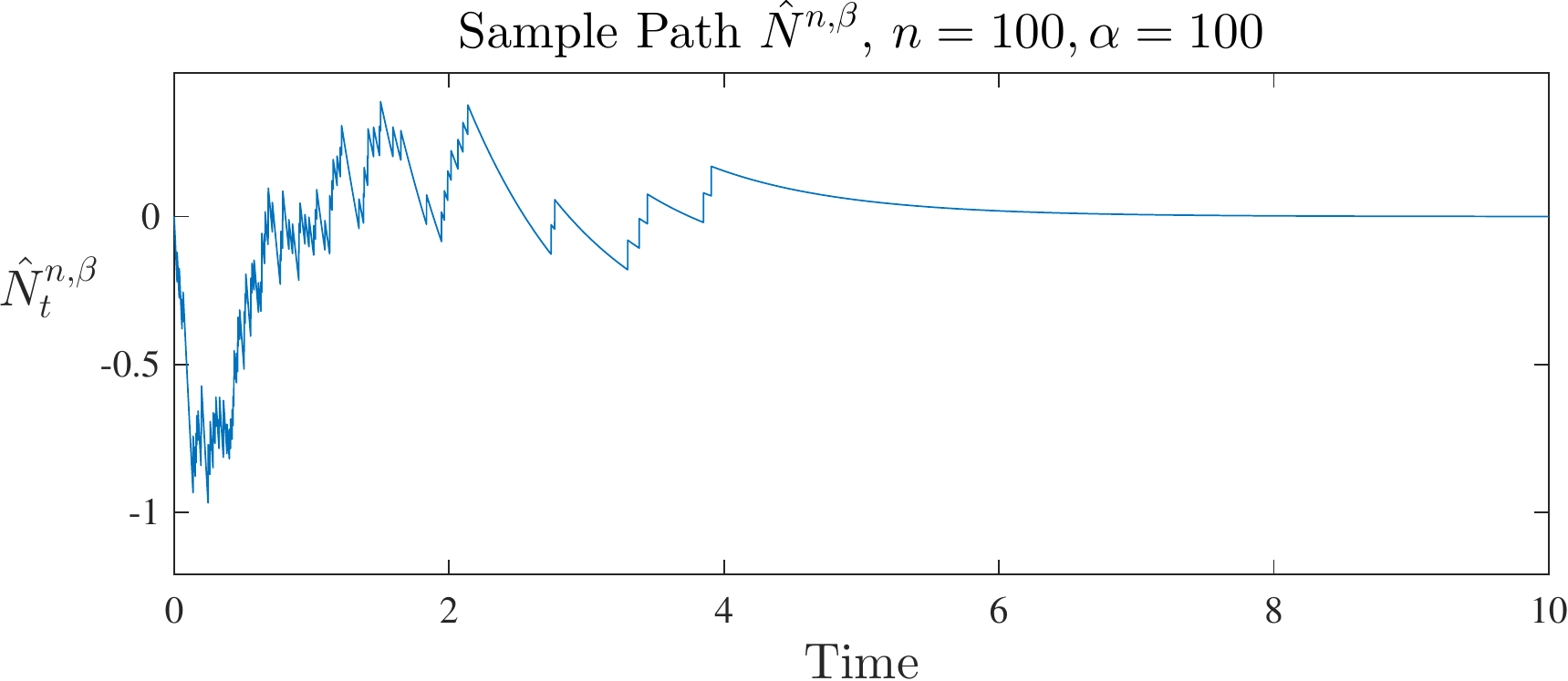} 
\end{subfigure}
\caption{CLT illustration $n=\alpha=10$ (left) and $n=\alpha=100$ (right)}
\label{fig: CLT 1}
\end{figure}

\begin{figure}[H]
\centering
\begin{subfigure}{.5\textwidth}
  \centering
    \includegraphics[width=1\linewidth]{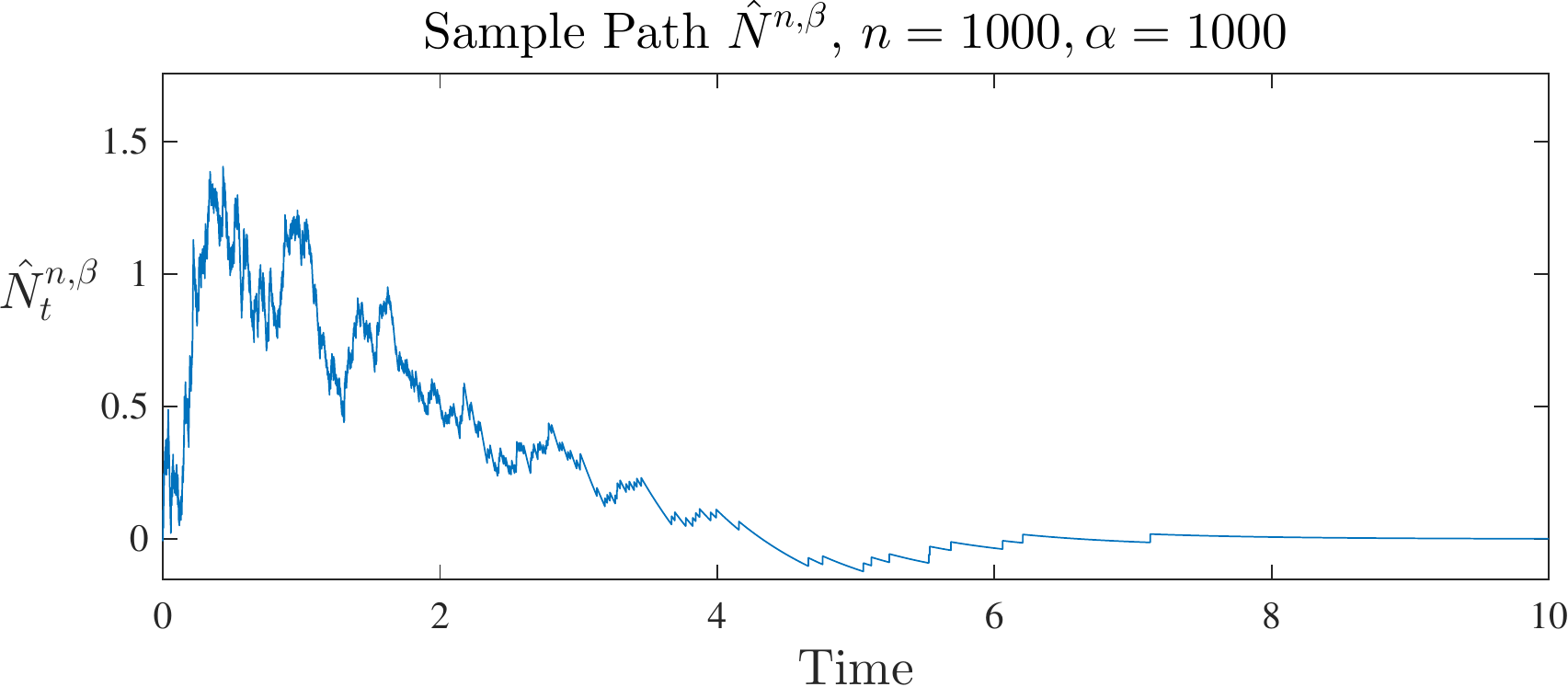} 
\end{subfigure}%
\begin{subfigure}{.5\textwidth}
  \centering
    \includegraphics[width=1\linewidth]{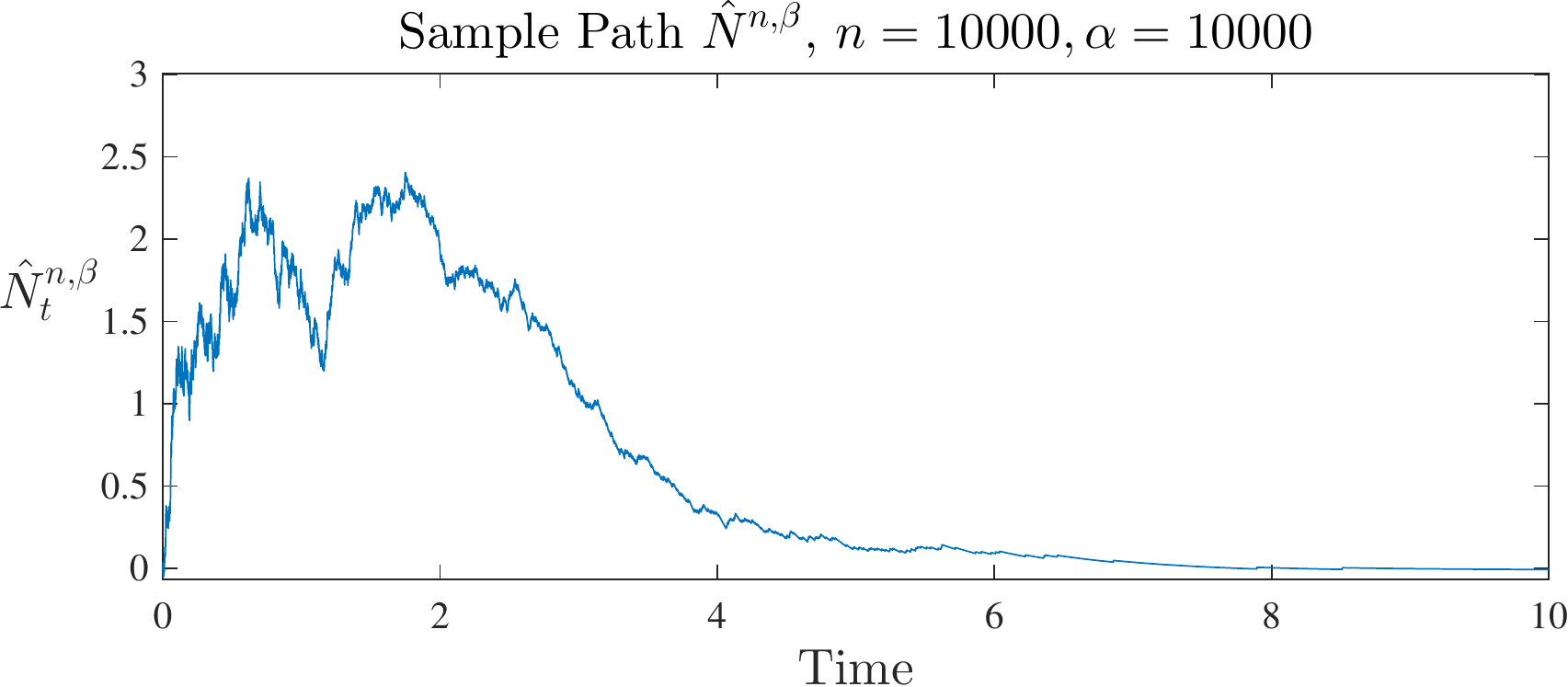} 
\end{subfigure}
\caption{CLT illustration $n=\alpha=1000$ (left) and $n=\alpha=10000$ (right)}
\label{fig: CLT 2}
\end{figure}

\autoref{fig: CLT 1} and \autoref{fig: CLT 2} illustrate how the process $\hat{N}^{n,\beta}$ converges to a continuous process, which fluctuates like a Gaussian martingale. 
We chose for the time scale $T=10$ to show that the quadratic variation $\langle \hat{N}\rangle_t$ of the limiting process $\hat{N}$ tends to a constant as $t \to \infty$. So these figures also illustrate the observations made in Remark~\ref{remark: Thm with beta}.

\section{Inclusion of recovery}\label{sec:recovery}

The process $N$ of \eqref{dN Bin} counts the number of defaults of companies (as one of the interpretations). After a default, a company disappears from the market. Alternatively, one might think of recovery of defaulted companies. In this section we present a few generalizations of previous results. As the proofs are similar to previous ones, but somewhat more involved, we only sketch them.

Supposing first that recovery happens at a constant rate $\mu$ per company and that Markov-modulation doesn't take place, we are dealing with a birth-death process $N$ whose semimartingale decomposition is, instead of \eqref{dN Bin}, now given by 
\begin{equation}\label{eq:bin+}
\diff N^n_t =\left(\lambda(n-N^n_t)-\mu N^n_t\right)\diff t + \diff M^n_t, \quad \Nn_0 = 0.
\end{equation} 
It is possible to show that $N^n$ is a Markov chain on $\{0,1,\ldots,n\}$ whose transition rates are $j\mu$ if $N$ jumps from $j$ to $j-1$ and $(n-j)\lambda$ if $N$ jumps from $j$ to $j+1$, whereas other transitions have rate zero. It follows that now $N_t$ has a Bin$(n,n\varrho_t)$ distribution, where $\varrho$ satisfies the differential equation
\[
\dot{\varrho}=\lambda(1-\varrho)-\mu\varrho,\, \quad \varrho_0=0.
\]
The solution to this equation is 
\[
\varrho_t=\frac{\lambda}{\lambda+\mu}(1-\exp(-(\lambda+\mu)t).
\]
To compute $\langle M^n\rangle_t=\langle N^n\rangle_t$, we first look at the  optional quadratic variation process $[N^n]$. As $[N^n]_t=\sum_{s\leq t}(\Delta N^n_s)^2$, and a nonzero $\Delta N^n_s$ is either plus or minus 1, which happens with rates $\lambda(n-N^n_t)$ and $\mu N^n_t$, respectively, it follows that $\frac{\diff}{\diff t}\langle M^n\rangle_t =  \lambda(n-N^n_t)+\mu N^n_t$.
\begin{proposition}\label{prop:EasyCLT2}
Let $\lambda,\mu>0$ be constants and let $N^n$ be given by \eqref{eq:bin+}.  Then the scaled and centered process 
\begin{equation*}
\hat{N}^n_t := n^{-1/2}(N^n_t-n\varrho_t)
\end{equation*}
converges weakly to the solution of the following SDE,
\begin{equation*}
\diff \hat{N}_t=-(\lambda+\mu) \hat{N}_t \diff t + \sigma(t)\diff B_t, \quad \hat{N}_0=0
\end{equation*} as $n \rightarrow \infty$.
Here $B$ is a standard Brownian motion and 
\[
\sigma(t)^2=\lambda-\frac{\lambda(\lambda-\mu)}{\lambda+\mu}(1-\exp(-(\lambda+\mu)t)=\lambda-(\lambda-\mu)\varrho_t.
\]
\end{proposition}
The proof of this proposition is similar to that of Proposition~\ref{Easy CLT N Bin}, so we only highlight the main differences. For the process $\hat{N}^n_t$ we now obtain
\[
\diff\hat{N}^n_t= - (\lambda+\mu) \hat{N}^n_t \diff t +  \diff \hat M^n_t,
\]
where the martingale $\hat M^n$ has quadratic variation satisfying (see above) 
\[
\frac{\diff}{\diff t}\langle \hat M^n\rangle_t=\frac{1}{n}\frac{\diff}{\diff t}\langle  M^n\rangle_t=\lambda-(\lambda-\mu)\frac{N_t}{n}\to\lambda- (\lambda-\mu)\varrho_t=\sigma^2(t), \mbox{ for $n\to\infty$}. 
\]
Obviously, the jumps of $M^n$ are negligible for large $n$. The remainder of the proof is as before.
\medskip\\
In \eqref{eq:bin+} the rates $\lambda$ and $\mu$ are constant. From now on we assume that regime switching will be present, so we have time varying rate $\lambda_t=\lambda\Top Z_{t-}$ as before and likewise, in similar notation, $\mu_t=\mu\Top Z_{t-}$. Hence for the Markov modulated case, we now have, instead of \eqref{MM Bin dNn},
\begin{equation}\label{eq:mmbin+}
\diff N^n_t = \left(\lambda\Top Z_t(n-N^n_t)-\mu\Top Z_{t-} N^n_t\right) \diff t + \diff M^n_t,
\end{equation} where $M^n$ is a martingale with respect to $\F = \{\cF^Y_t \vee \cF^Z_\infty, t\geq 0\}$. 

\begin{remark}
In principle, the recovery rate $\mu_t$ could depend on another Markov chain $\tilde Z$, leading to a seemingly more general model. But combining the chains $Z$ and $\tilde Z$ into a bivariate chain, would lead to a representation like \eqref{eq:mmbin+} again with the matrix $Q$ composed from the transition matrices of $Z$ and $\tilde Z$.
\end{remark}
One can then investigate the limit behaviour of the process $N^n$ given by \eqref{eq:mmbin+} for $n\to\infty$ together with rapid switching of the Markov chain. We confine ourselves to a generalization of Theorem~\ref{thm:mainthm2}. We use notation introduced in previous sections and self-evident analogies. We will need the function $\varrho$ satisfying
\begin{equation}\label{eq:varrhoinf}
\dot\varrho_t = \lambdai(1-\varrho_t)-\mu_\infty\varrho_t,\,\quad \varrho_0=0,
\end{equation}
where $\mu_\infty=\mu\Top\pi$, and the functions $\sigma_1$ and $\sigma_2$ as specified after the statement of the theorem.
\begin{theorem}\label{thm:mainthm3}
Let $\Nn$ be given by \eqref{eq:mmbin+} and $\varrho$ by \eqref{eq:varrhoinf}. Then the scaled and centered process $\hat{N}^{n,\beta}$ given by 
\begin{equation*}
\hat{N}^{n,\beta}_t :=n^{-1/2(1+(1-\beta)^+)}(N^{n}_t-n\varrho_t),
\end{equation*} converges weakly (as $n \rightarrow \infty$) to the solution of the following SDE 
\begin{equation}
\diff \hN_t = -(\lambdai+\mu_\infty) \hN_t \diff t  +\one_{\{\beta\leq 1\}} \sigma_1(t)\diff B^1_t+ \one_{\{\beta\geq 1\}} \sigma_2(t)\diff B^2_t, \quad \hN_0 = 0,
\end{equation} where $B^1$ 
and $B^2$ are independent Brownian motions. 
Alternatively, we have the representation
\[
\diff \hN_t = -(\lambdai+\mu_\infty) \hN_t \diff t  +
\left(\one_{\{\beta\leq 1\}} \sigma_1(s)^2 +
\one_{\{\beta\geq 1\}}\sigma_2(s)^2\right)^{1/2} \diff B_t,
\]
where $B$ is a standard Brownian motion.
\end{theorem}
We close with a few remarks on the proof. For the case $\beta=1$ it is along the same lines as the one for Theorem~\ref{MMbin Maintheorem}, but with more complicated expressions, although methodoligally there are hardly any changes. 
One now writes $\hNn_t = e^{-(\lambda+\mu)\Top \zeta^n_t}\Xn_t$, where $\Xn_t$ is given by an analog of \eqref{eq: 2nd SIE},
\begin{equation}\label{eq:xn+}
\Xn_t = - \int_0^t \Psi^n_s \diff n^{-1/2} \Zn_s + \int_0^t \Psi^n_s \diff n^{-1/2} \tilde{M}^n_s + \int_0^t e^{(\lambda+\mu)\Top \zeta^n_s} \diff n^{-1/2} \Mn_s,
\end{equation}
with in the present situation
\[
\Psi^n_s=e^{(\lambda+\mu)\Top\zeta^n_s}((1-\varrho_s)\lambda-\varrho_s\mu)\Top D.
\]
Another main difference is the quadratic variation of the bivariate martingale $\mathbf{M}^{\zeta,n}$.
One now obtains
\begin{align}
    \langle\mathbf{M}^{\zeta,n} \rangle_t & =  \int_0^t \begin{bmatrix}
    \Psi^n_s \tilde{V}^{n,*}_s (\Psi^n_s)\Top & 0 \\
    0 & e^{2(\lambda+\mu)\Top \zeta^n_s} V^{n,*}_s 
    \end{bmatrix} \diff s,\label{eq:qvm+}
\end{align}
where $\tilde{V}^{n,*}_s$ is as before, but
\begin{align*}
     & V^{n,*}_s = \lambda\Top \Zn_s (1 - n^{-1} \Nn_s)+\mu\Top \Zn_s n^{-1} \Nn_s \to \lambdai(1-\varrho_s)+\mu_\infty\varrho_s.
\end{align*}
As a consequence, the limit of $\langle\mathbf{M}^{\zeta,n} \rangle_t$ is not an expression as simple as before, but can still be computed explicitly (it only involves integration of exponential functions). For reasons of  brevity we just write
\begin{align*}
\int_0^t  
    \Psi^n_s \tilde{V}^{n,*}_s (\Psi^n_s)\Top  \diff s & \to \int_0^t e^{2(\lambdai+\mu_\infty)s}\Phi_s(\mathrm{diag}\{\pi\}D\Top+D\mathrm{diag}\{\pi\})\Phi_s\Top\diff s   \\
& =:\int_0^te^{2(\lambdai+\mu_\infty)s}\sigma_1(s)^2\diff s    ,
\end{align*}
where 
\[
\Phi_s=(1-\varrho_s)\lambdai\Top -\varrho_s\mu_\infty\Top,\quad\sigma_1(s)^2=\Phi_s(\mathrm{diag}\{\pi\}D\Top+D\mathrm{diag}\{\pi\})\Phi_s\Top
\]
and
\[
\int_0^t e^{2(\lambda+\mu)\Top \zeta^n_s} V^{n,*}_s\diff s \to \int_0^t e^{2(\lambdai+\mu_\infty)s}(\lambdai(1-\varrho_s)+\mu_\infty\varrho_s)\diff s=:\int_0^t e^{(2(\lambdai+\mu_\infty)s}\sigma_2(s)^2\diff s.
\]

\bibliographystyle{chicago-ff}
\bibliography{mmbinbibliography}

\end{document}